\documentclass[a4paper,11pt]{amsart}

\usepackage{amsmath, amsfonts, amssymb, amsthm, amscd, mathrsfs}
\usepackage{graphicx}
\usepackage{hyperref}
\usepackage{enumerate}
\usepackage[normalem]{ulem}
\usepackage{stmaryrd}
\usepackage{url}
\usepackage{color}
\usepackage[utf8]{inputenc}
\usepackage{tikz}
\usetikzlibrary{arrows}
\usetikzlibrary{patterns} 
\usepackage[a4paper,scale={0.72,0.74},marginratio={1:1},footskip=7mm,headsep=10mm]{geometry}

\numberwithin{equation}{section}

\newtheorem{theorem}{Theorem}
\newtheorem{lemma}{Lemma}
\newtheorem{proposition}[theorem]{Proposition}

\newtheorem{remark}[theorem]{Remark}

\newcommand{\cC}{{\ensuremath{\mathcal C}} }
\newcommand{\cD}{{\ensuremath{\mathcal D}} }

\newcommand{\cF}{{\ensuremath{\mathcal F}} }

\newcommand{\cJ}{{\ensuremath{\mathcal J}} }

\newcommand{\cM}{{\ensuremath{\mathcal M}} }
\newcommand{\cN}{{\ensuremath{\mathcal N}} }

\newcommand{\cV}{{\ensuremath{\mathcal V}} }
\newcommand{\cW}{{\ensuremath{\mathcal W}} }

\newcommand{\cZ}{{\ensuremath{\mathcal Z}} }

\newcommand{\gd}{\delta}

\newcommand{\gep}{\varepsilon}

\renewcommand{\tilde}{\widetilde}          
\DeclareMathSymbol{\leqslant}{\mathalpha}{AMSa}{"36} 
\DeclareMathSymbol{\geqslant}{\mathalpha}{AMSa}{"3E} 
\DeclareMathSymbol{\eset}{\mathalpha}{AMSb}{"3F}     
\newcommand{\dd}{\text{\rm d}}             
\newcommand{\ed}{\text{\rm e}}    		   

\newcommand{\R}{\mathbb{R}}
\newcommand{\C}{\mathbb{C}}
\newcommand{\Z}{\mathbb{Z}}
\newcommand{\N}{\mathbb{N}}

\newcommand\bP{\ensuremath{\mathrm{P}}}
\newcommand\bE{\ensuremath{\mathrm{E}}}

\renewcommand{\epsilon}{\varepsilon}

\newcommand{\8}{\infty}

\newenvironment{myenumerate}{
\renewcommand{\theenumi}{\arabic{enumi}}
\renewcommand{\labelenumi}{{\rm(\theenumi)}}
\begin{list}{\labelenumi}
{
\setlength{\itemsep}{0.4em}
\setlength{\topsep}{0.5em}
\setlength\leftmargin{2.45em}
\setlength\labelwidth{2.05em}
\setlength{\labelsep}{0.4em}
\usecounter{enumi}
}
}
{\end{list}
}

\renewenvironment{enumerate}{
\begin{myenumerate}}
{\end{myenumerate}}

\newcommand{\beq}{\begin{equation}}
\newcommand{\eeq}{\end{equation}}
\newcommand{\ba}{\begin{aligned}}
\newcommand{\ea}{\end{aligned}}

\begin{document}

\title{Evolution of a passive particle in a one-dimensional diffusive environment}

\author[François Huveneers]{François Huveneers}
\address[François Huveneers]{Ceremade, Universit\'e Paris-Dauphine, PSL, France}
\email{huveneers@ceremade.dauphine.fr}

\author[François Simenhaus]{Fran\c cois Simenhaus}
\address[François Simenhaus]{Ceremade, Universit\'e Paris-Dauphine, France}
\email{simenhaus@ceremade.dauphine.fr}

\date{\today}

\thanks{\emph{Acknowledgements: }We thank M.~Jara for useful discussions on the process $X$ evolving in a rough environment discussed in the Introduction. 
F.~H.\@ and F.~S.\@ are supported by the ANR-15-CE40-0020-01 grant LSD.
F.~S.\@ is supported by the ANR/FNS-16-CE93-0003 grant MALIN}


\maketitle

\begin{abstract} 
We study the behavior of a tracer particle driven by a one-dimensional fluctuating potential, 
defined initially as a Brownian motion, and evolving in time according to the heat equation.  
We obtain two main results. 
First, in the short time limit, we show that the fluctuations of the particle {become} Gaussian and sub-diffusive, with dynamical exponent $3/4$. 
Second, in the long time limit, we show that the particle is trapped by the local minima of the potential and evolves diffusively i.e. with exponent $1/2$. 
\end{abstract}


\section{Introduction}
Random walks in dynamical random environments have attracted a lot of attention recently. 
When the time correlations of the environment decay fast, several homogenization results have been obtained, 
see \cite{redig}\cite{blondel_hilario_teixeira} as well as references therein.
These results establish the existence of an asymptotic velocity for the walker (law of large numbers) 
and normal fluctuations around the average displacement (invariance principle). 
In the opposite extreme regime, when the environment becomes static, a detailed understanding of the behavior of the walker is available in dimension $d=1$, 
see e.g.~\cite{zeitouni} for a review. 

Diffusive environments in dimension $d=1$ constitute an intermediate case where memory effects are expected to become relevant, 
since correlations decay only with time as $t^{-1/2}$. 
Homogenization results are known when the walker drifts away ballistically, and escapes from the correlations of the environment
\cite{hilario_den_hollander}\cite{blondel_hilario_dos_santos}\cite{huveneers_simenhaus}\cite{salvi_simenhaus}. 
Recently, among other results, a law of large number with zero limiting speed was derived in \cite{hilario_kious},
for the position of a walker evolving on top of the symmetric simple exclusion process (SSEP). 
However, the question of the size of the fluctuations in such a case remains largely elusive,   
and several conflicting conjectures appear in the literature 
\cite{bohr_pikovsky}\cite{gopalakrishnan}\cite{nagar}\cite{avena_thomann}\cite{huveneers}. 
In a particular scaling limit, Gaussian anomalous fluctuations were shown to hold for a walker on the SSEP~\cite{jara_menezes}, see also below.

The aim of this paper is to advance our understanding on the fluctuations of a walker in an unbiased, one-dimensional, diffusive random environment, 
and to make clear that different behaviors may be observed depending on the time scales that we look at. 
For this, we introduce a new model where the evolution of the walker can be described in a fair bit of detail.  
To motivate properly the introduction of this new model, let us start with a question that, we believe, is of more fundamental interest.

\bigskip
\textbf{Motivation ---}
Let us consider a random potential $V = (V(t,x))_{t\ge 0, x\in\R}$ fluctuating with time according to the {stochastic heat equation}: 
\begin{equation}\label{eq: edwards wilkinson}
\begin{cases}
	V(0,x) \;  = \; B (x), \\
	\partial_t V (t,x) \;  = \; \partial_{xx} V(t,x) + \sqrt 2 \xi (t,x) 
\end{cases}
\end{equation}
where $B$ is a Brownian motion on $\R$, and where $\xi(t,x)$ is a space-time white noise. {We refer to \cite{hairer} for a gentle introduction to the stochastic heat equation.} 
The process $V$ is stationary, i.e.~$V(t,\cdot)$ is distributed as $B(\cdot)$ at all times $t\ge0$, 
and evolves diffusively in time, i.e.\@ the landscape described by $V(t,\cdot)$ in a box of size $L$ is refreshed after a time of order $L^2$. 
The potential $V$ solving~\eqref{eq: edwards wilkinson} can be obtained as the scaling limit of {the height function} of diffusive particle processes on the lattice, 
such as the SSEP, see e.g.~{Chapter 11 in} \cite{kipnis_landim}. 

We would like to consider a process $X = (X_t)_{t\ge 0}$ describing a particle driven by this potential in the overdamped regime: 
\begin{equation}\label{eq: X in rough field}
\begin{cases}
	X_0  \; = \;  0, \\
	\partial_t X_t  \; = \;   - \partial_x V (t,X_t) \; =: \; u (t,X_t)
\end{cases}	
\end{equation}
Since $V(t,\cdot)$ is rough, it is however not clear that the evolution equation~\eqref{eq: X in rough field} makes sense, 
and three natural questions arise: 
\begin{enumerate}
	\item Can the process $X$ be properly defined? 
	\item If yes, how does it behave on short time scales? 
	\item And how does it behave in the long time limit? 
\end{enumerate}
{We notice that, in the presence of an external random force, the analogous process in a static environment is well defined: 
There exists a process $X^{\text{st}}$ solving 
$$
	X^{\text{st}}_0 \; = \; 0, 
	\qquad 
	\dd X^{\text{st}}_t \; = \; - \partial_x V(0,X^{\text{st}}_t)\dd t + \dd \tilde B_t
$$ 
almost surely for all $t\ge 0$, where $(\tilde B_t)_{t\ge 0}$ is a Brownian motion independent of $B$ \cite{brox}\cite{hu shi}. 
Moreover, the long time behavior of $X^{\text{st}}$ is analogous to that of Sinai's random walk \cite{sinai}.}

{As far as we know}, the above questions were first addressed in \cite{bohr_pikovsky} by means of a heuristic fixed point argument. 
First, the authors conclude that if a process $X$ solves \eqref{eq: X in rough field}, it fluctuates sub-diffusively on short time scales: 
$X_{t+\Delta t} - X_t$ is typically of order $(\Delta t)^{3/4}$ for small $\Delta t$.
Second, the fluctuations of $X$ become (almost) diffusive on long time scales: $X_t$ is of order $(t \ln t)^{1/2}$ as $t$ grows large. 

The validity of these claims was analyzed in \cite{huveneers}, by means of numerical simulations and theoretical arguments 
{but, to the best of our knowledge, no rigorous proof has been provided so far.} 
The conclusion of \cite{huveneers} confirms the findings of \cite{bohr_pikovsky}, 
though the existence of a logarithmic correction in the long time behavior could not be ascertain.
Moreover, the analysis in \cite{huveneers} allows to view the process $X$ as the limit of well defined processes. 
We defer to Appendix~\ref{appendix} the few steps needed to recast the analysis developed in \cite{huveneers} into the present framework.

The occurence of two distinct behaviors, on short and long time scales, can be attributed to the {two} following mechanisms.
On short time scales, 
if the velocity field $u = - \partial_x V$ evolves much faster than the particle $X$, we can use the approximation
\begin{equation}\label{eq: short time approximation}
	X_{t+\Delta t} \; \simeq \; X_t + \int_0^{\Delta t} u (t+s,X_t) \dd s,
\end{equation}
that we expect to become exact in the limit $\Delta t \to 0$. 
Assuming moreover that the fast fluctuations of $u$ in the time interval $[t,t+\Delta t]$ are uncorrelated from $X_t$,
{we may further expect that the increments become stationary in the limit $\Delta t \to 0$, 
and we approximate $\int_0^{\Delta t} u (t+s,X_t) \dd s$ by $\int_0^{\Delta t} u (s,0) \dd s$. 
Therefore, since $u$ is Gaussian and $\bE(u(s,0) u(s',0)) = (4\pi|s-s'|)^{-1/2}$ for all $s,s'\ge 0$, we arrive at}
\begin{equation}\label{eq: sub diffusive delta t}
	\frac{X_{t+\Delta t} - X_t}{(\Delta t)^{3/4}} \quad \to \quad \cN(0,D)
\end{equation}
in law as $\Delta t \to 0$, with $D = 4/3\sqrt \pi$, and this result is consistent with the assumption that $u$ evolves much faster than $X$. 

{
However, we do not expect \eqref{eq: sub diffusive delta t} to be valid in the large time limit $\Delta t\to +\infty$, 
because $V$ imposes potential barriers that will trap the particle. 
Indeed, on the long run, we believe that the particle will move to the deepest local minimum of the potential that becomes reachable thanks to fluctuations, 
as it is the case for Sinai's walk. 
Since the potential barriers evolve diffusively, we expect the evolution of X to become eventually diffusive itself.}

In this paper, we introduce a simpler model, where these two mechanisms can be clearly exhibited, and the intuitive reasonings above made rigorous.
In particular, we will make clear how sub-diffusive behavior on {an initial} short time scale, and diffusive behavior on long time scales, can co-exist.
The simplification that we introduce is to remove all sources of fluctuations in the time evolution of the potential $V$: 
We still still consider a process $X$ solving \eqref{eq: X in rough field} with $V$ being now given by 
$$
	\begin{cases}
	V(0,x)\; = \; B(x),  \\
	\partial_t V(t,x) \; = \; \partial_{xx} V (t,x).
	\end{cases}
$$
Doing so, the potential $V$ becomes more and more regular as times evolves, and the sub-diffusive behavior \eqref{eq: sub diffusive delta t} only persists at $t=0$. 
On the other hand, trapping effects become more pronounced {as the time grows large}, leading to the eventual diffusive behavior of $X$. 
We also notice that the environment will no longer be stationary, 
and that our set-up bares some similarities with the random walks in cooling environment introduced recently in a series of papers \cite{avena_den_hollander,avena_chino,avena_chino_2}.

Finally, let us mention that the recent mathematical result in \cite{jara_menezes} provides a partial and indirect support to the conjecture~\eqref{eq: sub diffusive delta t}.
Indeed, the authors of \cite{jara_menezes} study a random walk $W^n = (W_{t}^n)_{0 \le t \le T}$, 
jumping on $\Z$ {at a rate proportional to $n$}, on top of the SSEP with a diffusion constant proportional to $n^2$. 
In the limit $n\to \infty$ and in the absence of drift, 
they derive that $W_{t}^n/\sqrt{n}$ converges to a sum of two Gaussian process,
with standard deviation at time $t$ proportional to $t^{1/2}$ and $t^{3/4}$ respectively. 
{As we explain in Appendix~\ref{appendix}, once properly rescaled, 
the processes $W^n$ converge to the putative process $X$ solving \eqref{eq: X in rough field}, but only on a time domain that shrinks to $0$ as $n\to\infty$.}

\bigskip
\textbf{Organization of this paper ---}
In Section~\ref{sec: definitions and results}, we define properly the model studied in this paper, and we state our two main results. 
The first one, Theorem~\ref{the: short time}, deals with the short time behavior of the passive particle, and is shown in Section~\ref{sec: short time}. 
The second one, Theorem~\ref{the: long time}, deals with its long time behavior, and is shown in Section~\ref{sec: long time}. 
Some informations on the behavior of the environment are collected in Section~\ref{sec: description of the environment}, 
and some intermediate results on the behavior of the zeros of the velocity field are gathered in Section~\ref{sec: zeros}. 


\section{Definitions and results}\label{sec: definitions and results}

We consider a one dimensional brownian motion $B = (B(x))_{x\in \R}$ and we define the random potential $V = (V(t,x))_{t\ge 0,x \in \R}$ by 
\begin{equation}\label{eq: u field}
	\begin{cases}
	V(0,x)\; = \; B(x) \quad \text{for all } x\in \R  \\
	\partial_t V(t,x) \; = \; \partial_{xx} V (t,x) \quad \text{for} \quad t>0,\;  x\in \R.
	\end{cases}
\end{equation}
Almost surely, the potential $V$ is well defined and analytic as a function of $t> 0$ and $x\in \R$. 
Indeed, let $ D  = \{t\in \C : \Re (t)>0\} \times \C$ and let us define the heat kernel as the complex function on $D$ such that
\begin{equation}\label{eq: gaussian kernel}
	(t,x) \; \mapsto \;  P_t(x) \; = \; \frac{\ed^{- \frac{x^2}{  4t}}}{\sqrt{ 4\pi t}} .
\end{equation}
The heat kernel is analytic as a function of the variables $t$ and $x$, for $(t,x)\in D$.
Moreover, almost surely, there exists $C> 0$ such that $|B(x)| \le C (|x|+1)$, see e.g.~\cite{evans}.
Therefore, we can define a function $V$ on $D$ by
$$
	V(t,x) \; = \; \int_\R P_t(x-y) B(y) \dd y , 
$$
it is analytic as a function of the variables $t$ and $x$, for $(t,x)\in D$, and it solves \eqref{eq: u field} for $t>0$ and $x\in \R$. 

Let $u = - \partial_x V$ be a velocity field. For all $t> 0$ and $x \in \R$, the representation 
\begin{equation}\label{eq: u representation}
	u(t,x) \; =  \; - \int_\R \partial_x P_t (x-y) B(y) \dd y \; = \; -\int_\R  P_t (x-y)  \dd B(y)
\end{equation}
holds.
We now introduce the process $X=(X_t)_{t\ge 0}$ that will be our main object of study, see also Fig.~\ref{fig: X and V}.
\begin{proposition}\label{pro: X well defined}
There exists a process $X = (X_t)_{t \ge 0}$ satisfying almost surely
\begin{equation}\label{eq:X}
	\begin{cases}
	X_0\; = \; 0 \\
	\partial_t X_t \; = \;  u (t,X_t) \quad \text{for} \quad t>0, 
	\end{cases}
\end{equation}
continuous on $\R_+$ and smooth on $\R_+^*$.
\end{proposition}

\begin{figure}[h]
    \centering
    \includegraphics[draft=false,height = 8cm,width = 12cm]{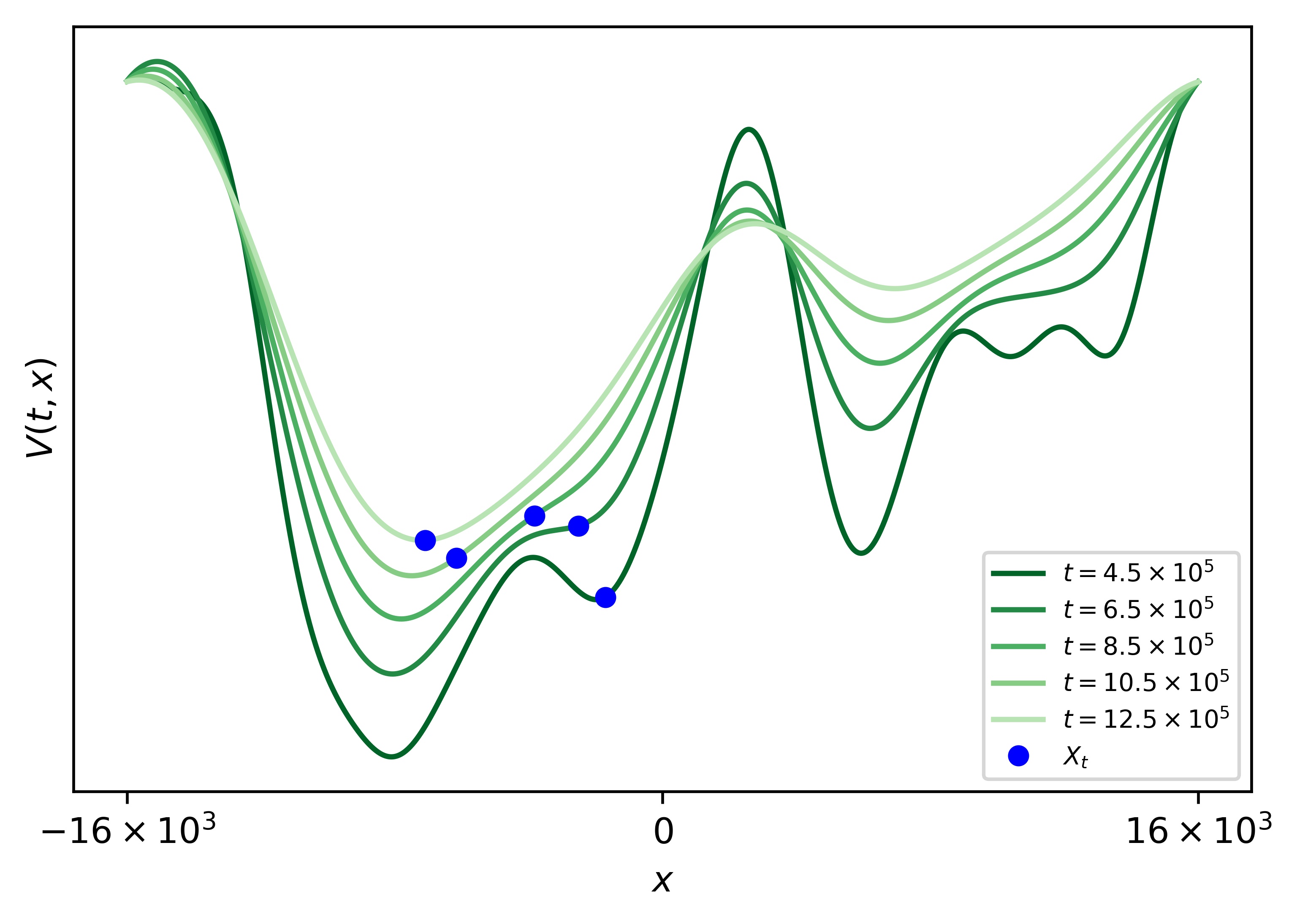}
    \caption{
    The process $X$ and the potential $V$ at different times. 
    In the long run, the particle $X$ sticks most of the time to a local minimum of the potential $V$, 
    as made precise in Theorem~\ref{the: long time}. 
    For the numerical simulation used to generate this plot, 
    we have assumed periodic boundary conditions and we have taken the initial condition $X_{t_0} = 0$ with $t_0=4\times 10^5$.}
    \label{fig: X and V}
\end{figure}

We want to show two results on the behavior of $X$.
The first one characterizes its short time behavior:  

\begin{theorem}\label{the: short time}
When the space of continuous function from $\R^+$ to $\R$ is endowed with the topology of uniform convergence on compact sets, the following convergence holds:
\begin{equation}\label{eq: short time}
	\left(\frac{X_{\theta T}}{T^{3/4}}\right)_{\theta\geq 0}
	\; \xrightarrow[T\to 0]{(law)} \; 
	\left(\int_0^\theta u (s,0) \dd s\right)_{\theta\ge 0}.
\end{equation}
\end{theorem}
  
We need to introduce some preliminary material to formulate our second result, dealing with the long time behavior of $X$.  
Some of the objects are illustrated on Fig.~\ref{fig: 0LRX}.
Given $t>0$, let us define the set of zeros of the field $u$ at time $t$: 
$$
	\mathcal Z_t \; = \; \{ x \in \R : u(t,x) = 0\}\, .
$$
Let us distinguish attractive, or stable zeros, from repulsive, or unstable ones: 
$$
	\begin{cases}
	\mathcal Z_t^{\mathrm s} \; = \;  \{ x \in \R : u(t,x) = 0, \partial_x u (t,x) < 0\}\, , \\
	\mathcal Z_t^{\mathrm u} \; = \;  \{ x \in \R : u(t,x) = 0, \partial_x u (t,x) > 0\} .
	\end{cases}
$$
We may also observe zeros that are nor stable nor unstable, say neutral: 
$$
	\mathcal Z_t^{\mathrm n} \; = \;  \{ x \in \R : u(t,x) = 0, \partial_x u (t,x) = 0\} .
$$ 

The next lemma allows to ``trace back'' a zero at time $t$ up to time $0$: 
\begin{lemma}\label{lem: definition r}
	Almost surely, for all $t>0$ and all $x \in \mathcal Z_t^{\mathrm s} \cup \mathcal Z_t^{\mathrm u}$, 
	there exists a unique continuous function $r_{(t,x)}:[0,t]\to \R$ such that for all $0<s\leq t$,
	$$
		u(s,r_{(t,x)}(s))=0
	$$
	and actually $r_{(t,x)} (s) \in \mathcal Z_s^{\mathrm s}$ if $x \in \mathcal Z_s^{\mathrm s}$ 
	and $r_{(t,x)} (s) \in \mathcal Z_s^{\mathrm u}$ if $x \in \mathcal Z_s^{\mathrm u}$
	(and thus in particular $\partial_x u(s, r_{(t,x)} (s)) \ne 0$).
	The function $r_{(t,x)}$ is smooth on $]0,t[$ and for all $0<s \leq t$,
	\beq\label{eq:deriv r}
		\partial_s r_{(t,x)}(s)
		\; = \; 
		- \frac{\partial_{s} u(s,r_{(t,x)}(s))}{\partial_xu(s,r_{(t,x)}(s))} 
		\; =\; 
		- \frac{\partial^2_{xx} u(s,r_{(t,x)}(s))}{\partial_xu(s,r_{(t,x)}(s))}.
	\eeq
\end{lemma}

Once properly rescaled, the long time behavior of $X$ is described by the limiting process $Z = (Z_t)_{t\ge 0}$ introduced in the following proposition: 

\begin{proposition}\label{pro: Z process}
	There exist unique processes $L = (L_t)_{t\ge 0}$ and $R = (R_t)_{t\ge 0}$ such that $L_0=R_0 = 0$ and, almost surely, for all $t>0$, 
	\begin{equation}\label{eq: def Lt Rt}
		\begin{cases}
		L_t \; = \; \max \{x \in \mathcal Z_t^{\mathrm s} \cup \mathcal Z_t^{\mathrm u} : r_{(t,x)}(0) < 0\},\\
		R_t \; = \; \min \{x \in \mathcal Z_t^{\mathrm s} \cup \mathcal Z_t^{\mathrm u} : r_{(t,x)}(0) > 0\}.
		\end{cases}
	\end{equation}
	Moreover, almost surely, for all $t>0$, one and only one of the following events occur
	\begin{equation}\label{eq: alternative stable unstable}
		\big( L_t \in \mathcal Z_t^{\mathrm s} \quad\text{and}\quad R_t \in \mathcal Z_t^{\mathrm u} \big) 
		\qquad \text{or} \qquad
		\big( L_t \in \mathcal Z_t^{\mathrm u} \quad\text{and}\quad R_t \in \mathcal Z_t^{\mathrm s} \big).
	\end{equation}
	We can thus define a process $Z = (Z_t)_{t\ge 0}$ by $Z_0 = 0$ and
	$$
		\begin{cases}
		Z_t \; = \; L_t \quad \text{if}\quad L_t \in  \mathcal Z_t^{\mathrm s}\\ 
		Z_t \; = \; R_t \quad \text{if}\quad R_t \in  \mathcal Z_t^{\mathrm s}
		\end{cases}
	$$
	for $t > 0$. The following properties of $Z$ hold: 
	
	1. Almost surely, $Z$ is càdlàg. 
	
	2. Almost surely, $Z$ is discontinuous at some time $t > 0$ if and only if $Z_{t-} \in \mathcal Z^{\mathrm n}_{t}$.
	
	3. Almost surely, for any compact interval $I\subset \R_+^*$, the number of discontinuities of $(Z_t)_{t\in I}$ is finite and $(Z_t)_{t\in I}$ is smooth away from the jumps.
	
	4. $(Z_t)_{t\ge 0} = (T^{-\frac12}Z_{T t})_{t\ge 0}$ in law for all $T > 0$.
	
	5. The variable $Z_1$ has a bounded density and there exists $c>0$ such that, for all $z \ge 0$,
	$$
		{ c} \, \ed^{- z/c} \; \le \; \bP(|Z_1| \ge z) \; \le \; { \frac1c} \, \ed^{-cz}.
	$$ 
\end{proposition}

\begin{remark}
\label{rq:aussiLR}
By a straightforward adaptation of the proof, items $1$ to $5$ above can be shown to hold as well with $L$ or $R$ in place of $Z$.
\end{remark}

\begin{figure}[h]
    \centering
   	\includegraphics[draft=false,height = 8cm,width = 14cm]{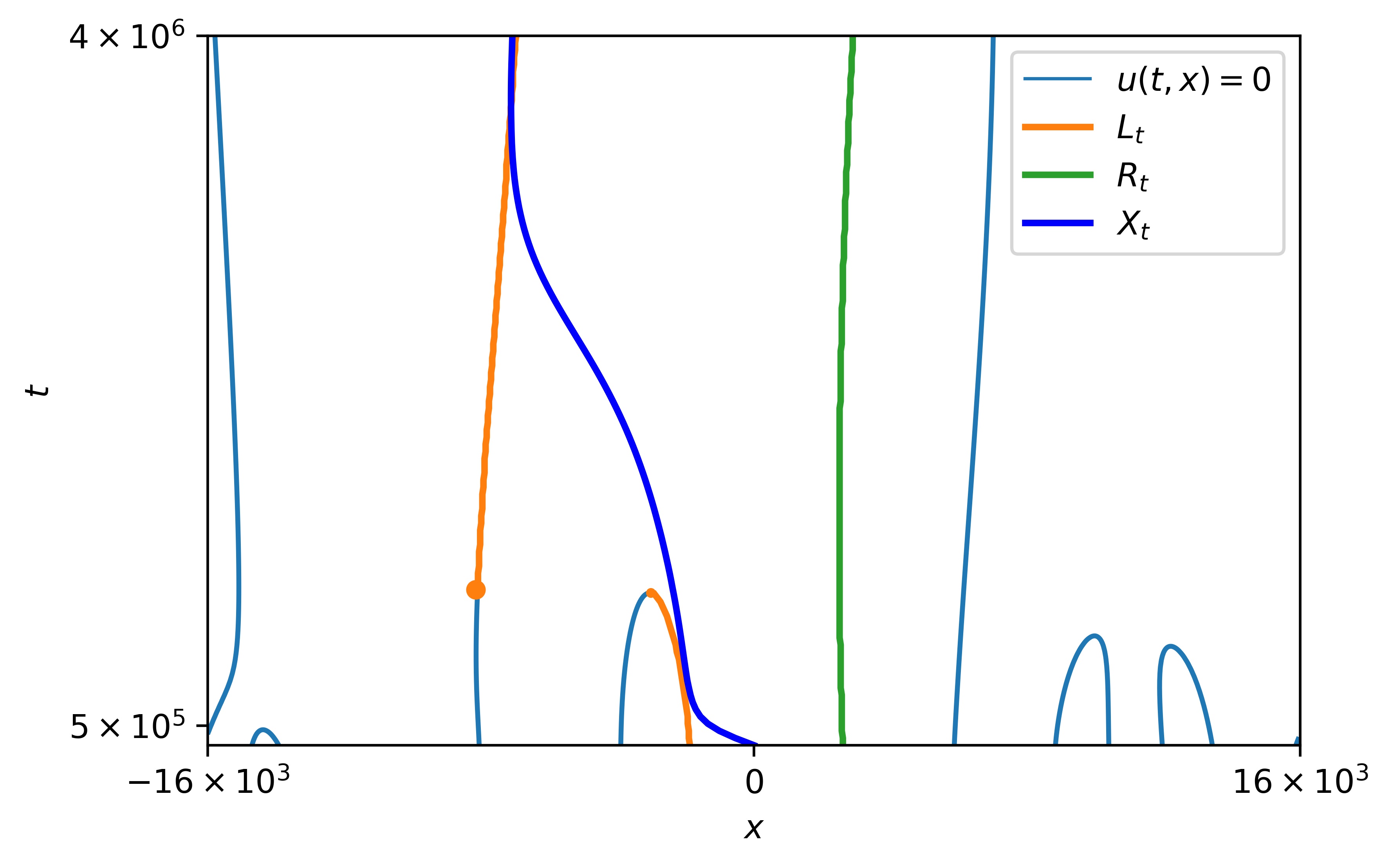}
    \caption{
    The processes $L$, $R$ and $X$, as well as the zeros of the velocity field $u$. 
    The realization of the environment is the same as on Fig.~\ref{fig: X and V}.}
    \label{fig: 0LRX}
\end{figure}

We now come to our result on the long time behavior of $X$: 

\begin{theorem}\label{the: long time}
When the space of càdlàg function from $\R^+$ to $\R$ is endowed with the Skorokhod's $\cM_1$ topology for the convergence on compact sets, the following convergence holds:
\begin{equation}\label{eq: long time proba}
	\frac{1}{T^{1/2}} \left( X_{\theta T} - Z_{\theta T} \right)_{\theta\ge 0}
	\; \xrightarrow[T\to +\8]{(probability)} \; 
	0
\end{equation}
and thus in particular, by the scaling relation in item 4.\@ in Proposition~\ref{pro: Z process}, 
\begin{equation}\label{eq: long time law}
	\left(\frac{X_{\theta T}}{T^{1/2}} \right)_{\theta\geq 0}
	\; \xrightarrow[T\to +\8]{(law)} \; 
	(Z_\theta)_{\theta\ge 0}.
\end{equation}
\end{theorem}

\begin{remark}
	Since the process $(X_{\theta T})_{\theta \ge 0}$ is continuous and the process $(Z_{\theta T})_{\theta \ge 0}$ has jumps, 
	it is not possible to obtain the convergence in the Skorokhod's $\cJ$ topology. 
	Let us remind the definition of the $\cM_1$ topology, see~\cite{whitt}. 
	Let $D([0,1],\R)$ be the space of real càdlàg functions on $[0,1]$. 
	For $f \in D([0,1],\R)$, the completed graph of $f$ is defined as
	$$
		\mathcal G(f) \; = \; \{(t,x) \in [0,1]\times \R : x \in [f(t-)\wedge f(t),f(t-)\vee f(t)] \}
	$$
	with $f(t-) = \lim_{s\to t, s < t}f(s)$.
	We define an order on $G(f)$ as follows
	$$
		(t,x) \le (s,y) 
		\quad \Leftrightarrow \quad
		t < s \quad \text{or} \quad t=s \text{ and } |x-f(t-)| \le |y-f(s-)|.
	$$
	A parametrization of $\mathcal G(f)$ is defined to be a continuous map $\varphi:[0,1]\to \mathcal G(f)$
	such that $\varphi(0) = (0,f(0))$, $\varphi(1) = (1,f(1))$ and $\varphi$ is non-decreasing for the above order. 
	The set of parametrizations of $\mathcal G(f)$ is denoted $\Pi(f)$.
	The $\cM_1$ distance between two elements of $D([0,1],\R)$ is defined as
	$$
		d_{\cM_1} (f,g) 
		\; = \; 
		\inf_{\varphi \in \Pi(f),\psi \in \Pi(g)} 
		\left\{\sup_{0 \le \tau \le 1}|\varphi(\tau) - \psi(\tau)|_\infty\right\}
	$$
	where $|(t,x)|_\infty = \max\{|t|,|x|\}$. {The definition is completely analogous on any other compact interval of $\R$.}
\end{remark}


\section{Description of the environment}\label{sec: description of the environment}

We establish here several features of the environment $u$, that will be used throughout this text.  
We first show the scaling property~\eqref{eq: scaling} below that will, among other things, play a key role in establishing Theorem~\ref{the: long time}.
Second, we construct a grid of space-time points such that $u$ keeps the same sign on some time interval around each of these points, 
see Proposition~\ref{prop: main step exp decay} below as well as the subsequent constructions. 
This grid allows to derive a priori bounds on the processes $X$, $L$, $R$ and $Z$, that depend only on the sign of the the velocity field $u$.
Third, we obtain estimates on the supremum of $u$ and its derivatives, see Lemma~\ref{lem: upper bounds on u and del u} below.  
These estimates will be mainly needed in the proof of Proposition~\ref{pro: X well defined}. 

\bigskip
\textbf{Scaling property.}  
For any $\alpha > 0$, 
\begin{equation}\label{eq: scaling}
	(u (t,x))_{t,x} \; \stackrel{(law)}{=} \; (\alpha^{1/4} u (\alpha t , \alpha^{1/2}x))_{t,x}.
\end{equation}
Indeed, both fields are Gaussian, centered and have the same covariance and, from the representation of the field $u$ in~\eqref{eq: u representation}, we compute
\begin{equation}\label{eq: covariance}
	\bE (u(t,x) u(s,y)) 
	\; = \;
	\int_\R P_t(x-z) P_s(z-y) \dd z 
	\; = \; 
	P_{t+s}(x-y) , 
\end{equation}
and $\alpha^{1/2} P_{\alpha (t+s)}(\alpha^{1/2}(x-y)) = P_{t+s}(x-y)$ from \eqref{eq: gaussian kernel}. 

\bigskip
\textbf{Sign of the field $u$.}
Given $\ell >0$, let 
\begin{equation}\label{eq: D+}
	D(\ell) \; = \; \{\exists y : |y| \le \ell/2 \text{ and }  \forall s \in [1/2,1],  u(s,y) > 0 \}.
\end{equation}
The following proposition provides a control on the probability of $D(\ell)$:
\begin{proposition}\label{prop: main step exp decay}
	There exists $C>0$ such that, for all $\ell>0$,
	$$
		\bP\left(\overline{D(\ell)}\right)
		\; \le \; 
		\frac1C \ed^{-C \ell}.
	$$
\end{proposition}

\begin{remark}\label{rem: extends proposition}
By symmetry of $u$, by translation invariance and by the scaling property~\eqref{eq: scaling}, 
we deduce from the above proposition that there exists $C>0$ such that,
for any $x\in\R$ and any $t>0$, 
$$
	\bP\left(\overline{\{ \exists y : |y-x| \le \ell \sqrt t /2 \text{ and } \forall s \in [t/2,t], \pm u (s,y) > 0 \}}\right)
	\; \le \; 
	\frac1C \ed^{-C \ell}.
$$
\end{remark}

\begin{proof}
We divide the proof into several steps. 

\bigskip

1. Given a compact interval $I \subset \R$ and some $\alpha > 0$,
\begin{equation}\label{eq: bound sup phi}
		\bP \left( \sup \left\{ \int_I \varphi (x) \dd B(x), \varphi\in\cC^1(I), \| \varphi\|_{\cC^1} \le 1 \right\} > \alpha \right) 
		\; \le \; 
		\frac{ 1 + 2|I|^{3/2}}{\alpha}
\end{equation}
with $\|\varphi\|_{\cC^1} = \max_{x\in I}\{|\varphi(x)|+|\varphi'(x)|\}$.

Indeed, let $I = [a,b]$ be some compact interval, and let  $\varphi \in \cC^1(I)$. 
An integration by parts yields
$$
	\int_I \varphi (x) \dd B(x) \; = \; \varphi(b) (B(b) - B(a)) - \int_I \varphi'(x) (B(x) - B(a)) \dd x.
$$
Hence, 
$$
	\sup_{\varphi : \|\varphi\|_{\cC^1} \le 1} \int_I \varphi (x) \dd B(x) 
	\; \le \;
	|B(b) - B(a)| + \int_I |B(x) - B(a)| \dd x
$$
and, by Markov inequality, for any $\alpha > 0$, 
\begin{align*}
	\bP\left( \sup_{\varphi : \|\varphi\|_{\cC^1} \le 1} \int_I \varphi (x) \dd B(x) \; > \;  \alpha \right) 
	\; &\le \; \frac1\alpha \bE\left( |B(b) - B(a)| + \int_I |B(x) - B(a)| \dd x \right) \\
	\; &\le \; \frac{1 + 2(b-a)^{3/2}}{\alpha}.
\end{align*}

\bigskip

2. We introduce some definitions and notations. 
Let $\ell_\star \ge 1$. For $k \in \Z$, we define the points $x_k = k \ell_\star$ and the intervals 
$$
	I_k \; = \; \left[ x_k - \ell_\star/2 , x_k + \ell_\star/2 \right] \, ,
$$
as well as the variables 
$$
	h_k \; = \; \frac1{\ell_\star^{3/2}}\sup\left\{ \int_{I_k} \varphi(x) \dd B(x), \; \varphi\in\cC^1(I_k), \| \varphi\|_{\cC^1} \le 1\right\}.
$$
Let also $n_0\in\N^*$ and let us define the variables 	
$$
	H_k \; = \; 0 \vee \left\lfloor \frac{\ln h_k}{n_0} \right\rfloor .
$$
We observe that, by \eqref{eq: bound sup phi}, $\bP(H_k = 0)$ goes to $1$ as $n_0$ goes to infinity, uniformly in $\ell_\star$. 

\bigskip

3. Given $k\in\Z$ and $t>0$, let us define 
$$
	\tilde u(t,x_k) \; = \; -\int_{I_k} P_t(x_k - y) \dd B(y).
$$
We want to control the difference between $\tilde u (t,x_k)$ and $u(t,x_k)$. 
For this, let us introduce the variables 
$$
	\eta_k \; = \; 
	\left\{ 
	\begin{array}{ll} 
		1 & \text{if}\;\; H_{k-j} \le |k-j|\;\quad \forall j \in \Z, \\
		0 & \text{otherwise} 
	\end{array}
	\right.
$$
We claim that, given $n_0$, for all $\ell_\star$ large enough, and for all $k$ such that $\eta_k = 1$, 
\begin{equation}\label{eq: troncature}
		\sup_{1/2\le t \le 1} |u(t,x_k) - \tilde u(t,x_k)| < 1 . 
\end{equation}

Let us show \eqref{eq: troncature}. By translation invariance, it suffices to consider the case $k=0$. For all $t>0$, 
$$
	u(t,0) - \tilde u(t,0) \; = \; -\sum_{j\in\Z\backslash\{0\}} \int_{I_j} P_t(y) \dd B(y).
$$
Since $\eta_0=1$, it holds that $h_j \le \ed^{|j|n_0}$ for all $j\in\Z\backslash\{0\}$. 
Moreover, there exist $C,c>0$ such that $\| -P_t(\cdot) \|_{\cC^1(I_j)}\leq C \ed^{-c |j|^2\ell_\star^2}$ for all $j\in\Z$, so that we finally obtain
$$
	\sup_{1/2\le t \le 1} |u(t,0) - \tilde u(t,0)| 
	\; \le \; 
	\sum_{j\in\Z\backslash\{0\}} C \ell_\star^{3/2} \ed^{-c |j|^2\ell_\star^2}  \ed^{|j|n_0},
$$
and this becomes smaller than $1$ for $\ell_\star$ large enough.

\bigskip 

4. For all $n_0$ large enough, for all $\ell_\star\ge 1$, and for all $N\in \N$,
\begin{equation}\label{eq: exponential decay bad configurations}
	\bP(|\eta^N|_1 < N) \; \le \; \ed^{-N},
\end{equation}
where $\eta^N = (\eta_{-N}, \dots , \eta_N)$ and $|\eta^N|_1 = \sum_{k = -N}^N \eta_k$.

Indeed, on the event
$$
	E_N \; = \; \{H_j \leq \left|\ |j|-N \ \right| \text{ for all }|j|>2N \},
$$
it holds that
$$
	\{k \in \Z : |k|\le N, \eta_k = 0\} \;\subset\; \bigcup_{ -2N \leq j \leq 2N} \{j-(H_j-1),\cdots, j + (H_j-1)\}
$$
with the convention $\{a,\cdots,b\}=\emptyset$ if $b<a$. Therefore 
$$
	|\{k \in \Z : |k|\le N, \eta_k = 0\}| \;\leq\; \sum_{ -2N \leq j \leq 2N}  \left[(2H_j-1)\vee 0\right] \;\leq\;  \sum_{ -2N \leq j \leq 2N}  2H_j.
$$
We thus obtain
\begin{equation}\label{eq: exp decay split two terms}
	\bP(|\eta^N|_1 < N)\; \leq\;  \bP(\sum_{ -2N \leq j \leq 2N}  2H_j> N+1) + \bP(\overline{E_N}).
\end{equation}
For the second term, since by \eqref{eq: bound sup phi}, for any $n\in\N$, $\bP(H_j\ge n) \le 3\ed^{- n_0 n}$ it holds that for $n_0$ large enough and all $N\geq 1$,
\begin{equation}\label{eq: exp decay split 2d term}
	\bP(\overline{E_N}) \; \leq \;  \sum_{|j|>2N} \bP(H_j > \left||j|-N \right| ) \; \leq \; \frac12 \ed^{-N}.
\end{equation}
For the first one, as the variables $(H_j)_{j\in\Z}$ are i.i.d.\@, we obtain
\begin{equation}\label{eq: exp decay split 1st term}
	\bP(\sum_{ -2N \leq j \leq 2N}  2H_j> N+1) \;\leq\; \ed^{-2(N+1)} \bE\left(\ed^{ 4H_0}\right)^{4N+1}.
\end{equation}
Finally, using once again that for any $n\in\N$, $\bP(H_j\ge n) \le 3\ed^{- n_0 n}$, 
one can choose $n_0$ large enough so that for all $N\ge 1$, the last term in \eqref{eq: exp decay split 1st term} is smaller that $\ed^{-N}/2$.
Inserting \eqref{eq: exp decay split 2d term} and \eqref{eq: exp decay split 1st term} into \eqref{eq: exp decay split two terms} yields the result.

\bigskip

5. There exists $p >0$ such that, for all $\ell_\star\ge 1$ and for all $k\in\Z$,
\begin{equation}\label{eq: p features}
	\bP\left(\inf_{1/2\le t \le 1}  \tilde u(t,x_k) \;>\; 1\right) \; = \; p.
\end{equation}

Again, to show this, it suffices to consider the case $k=0$.
We observe that, almost surely, $\inf_{1/2\le t \le 1}  \tilde u(t,0)$ is a continuous function of $\ell_\star$,
converging to $\inf_{1/2\le t \le 1} u(t,0)$ as $\ell_\star \to \infty$. 
Therefore, it is enough to establish the result for any fixed $\ell_\star \ge 1$ and for $u$ instead of $\tilde u$.
This last case can be handled with exactly the same proof, and we let $\ell_\star \ge 1$. 
Since the variables $(\tilde u(t,0))_{t>0}$ are positively correlated, for any $0< a < b$, it holds that 
\begin{multline*}
	\bP\left(\inf_{t\in[a,b]} \tilde u(t,0) \; > \; 1/2 \right)\\
	\; \ge \; 
	\bP\left(\inf_{t\in[a,(a+b)/2]} \tilde u(t,0) \; > \; 1/2 \right)
	\bP\left(\inf_{t\in[(a+b)/2,b]} \tilde u(t,0) \; > \; 1/2 \right).
\end{multline*}
Let us assume that $\bP\left(\inf_{1/2\le t \le 1}  \tilde u(t,0) > 1/2\right) = 0$. 
By the above property, we find a sequence of nested closed intervals $(I_n)_{n\ge 1}$ with $I_1 = [1/2,1]$ and $|I_n| = 2^{-n}$ such that, for all $n\ge 1$,
$\bP\left(\inf_{t \in I_n}  \tilde u(t,0) > 1/2\right) = 0$.
Let $t_0 \in \cap_{n\ge 1} I_n$. 
Since $\tilde u(\cdot,0)$ is continuous almost surely in $t_0$, 
$$
	\{ \tilde u(t_0,0)>1/2\} \; \stackrel{a.s.}{=} \; \bigcup_{n\ge1}\{ \tilde u(t,0)>1/2,t\in I_n\}, 
$$
and thus $\bP( \tilde u(t_0,0) > 1/2) = 0$. 
This is a contradiction.

\bigskip

6. We start now the proof of the proposition itself. 
Let $p$ be the constant featuring in \eqref{eq: p features},
let $n_0$ be large enough so that, for all $\ell_\star \ge 1$, $\bP(H_0) \ge 1-p/2$ and such that \eqref{eq: exponential decay bad configurations} holds, 
and finally let $\ell_\star$ be large enough so that \eqref{eq: troncature} holds. 

Let $N\in \N^*$, and let $\ell = (2N+1) \ell_\star$.
Since the events $D(\ell)$ are increasing with $\ell$, it suffices to show the proposition for $\ell$ of this type.
We start with 
\begin{align}
	\bP\left(\overline{D(\ell)}\right)
	\; &\le \; 
	\bP \left(\inf_{1/2 \le t \le 1} u (t,x) \; \le \; 0, \ \forall x \in[-\ell/2,\ell/2] \right)\nonumber\\
	\; &\le \; 
	\bP \left( \bigcap_{-N\le i \le N}\left\{ \inf_{1/2 \le t \le 1} u (t,x_i) \; \le \; 0 \right\}\right).\label{eq: 1st rough bound}
\end{align}
Let us denote by $A$ the event featuring in the right hand side of this last expression. 
From \eqref{eq: exponential decay bad configurations}, we obtain 
\begin{equation}\label{eq: some exp bound A}
	\bP(A) \; \le \; \sum_{\stackrel{\sigma\in\{0,1\}^{2N+1},}{|\sigma|_1 \ge N}} \bP(A\; |\; \eta^N = \sigma) \bP(\eta^N = \sigma) \; + \; \ed^{-N}.
\end{equation}
Moreover from \eqref{eq: troncature}, we deduce that for all $i\in \Z$,
$$
	\left\{ \inf_{1/2 \le t \le 1} u (t,x_i) \; \le \; 0 \right\} \cap \{\eta(i) = 1\}
	\;\subset\;
	\left\{ \inf_{1/2 \le t \le 1} \tilde u (t,x_i) \; \le \; 1 \right\}.
$$
Given $\sigma$ such that $|\sigma|_1 \ge N$, 
we define $-N\leq i_1<\cdots <i_N\leq N$ to be the $N$ distinct smallest indexes such that  for all $1\leq j \leq N$, $\eta_{i_j}=1$. 
We denote by $J$ the complementary set of the $(i_k)_{1\leq k \leq N}$ in $\Z$.
The event $\{\eta^N=\sigma\}$ can be written as
$$
	\left(\bigcap_{1\leq k\leq N}\{H_{i_k}=0\}\right) \cap \{(H_j)_{j\in J}\in B\}
$$
with $B$ some suitable event in $\N^{J}$.
Therefore, as the $(H_j)_{j\in \Z}$ are i.i.d., we obtain
$$
	\bP(A |\eta^N = \sigma)\; \le\;
	\prod_{k=1}^N \bP\left(\inf_{1/2 \le t \le 1} \tilde u (t,x_{i_k}) \; \leq\; 1\Big |H_{i_k} = 0\right) 
	\; \le \; 
	\left(\frac{1-p}{1-p/2}\right)^N,
$$
and the proof follows by inserting this bound into \eqref{eq: some exp bound A}.
\end{proof}

In the following we make use of Proposition \ref{prop: main step exp decay} to give a property of the environment that we will use repeatedly till the end of the article. 
Let $K \ge 1$ be a constant that will be fixed below.
Given $k\geq 0$ and $\alpha \geq 1$, we define a finite family of space-time boxes covering 
$B(k,\alpha)=[0,2^k]\times \left[-K \alpha \sqrt{2^k}, +K \alpha \sqrt{2^k}\right]$ in the following way: 
For all $n\geq 0$, we define
$$
	t_n(k)=2^{k-n}  \quad \text{ and }\quad
	\ell_n(k,\alpha)= (\alpha +n^2)\sqrt {t_n(k)},
$$
and also the space intervals  
$$
	I_{n,j}(k,\alpha)=\left[j\ell_n(k,\alpha),(j+1)\ell_n(k ,\alpha)\right], \quad j\in \Z.
$$
We note $J_{n}(k,\alpha)$ the set of $j$ such that $I_{n,j}(k,\alpha)$ intersects $[ -K \alpha \sqrt{2^k}, K \alpha \sqrt{2^k}]$, 
so that $B(k,\alpha)$ is covered by the family of boxes $[t_{n+1}(k,\alpha),t_{n}(k,\alpha)]\times I_{n,j}(k,\alpha)$, 
for $n\geq 0$ and $j\in J_{n}(k,\alpha)$.

From now on we fix $K$ large enough so that for all $\alpha\ge 1$ and $k\ge 0$,
\beq \label{eq:defn0}
	\sum_{n\geq 0} 2 \ell_n(k,\alpha) \;\leq\; \frac{K }{3}   \alpha\sqrt{t_0(k)}.
\eeq
The reason for defining $K$ in this way will become clear later. 
For all $n\geq 0$ and $j\in J_{n}( k,\alpha)$, we consider the event
$$
	E_{n,j}(k,\alpha) 
	\; = \;
	\{\exists y_1,y_2 \in I_{n,j}(k,\alpha):\ \forall s \in [t_{n+1}(k ,\alpha),t_{n}(k{,\alpha})], u(s,y_1) > 0 \text{ and }u(s,y_2) < 0  \}.
$$
and we define
\beq
	\label{eq: def G}
	G(k,\alpha)\; = \; \bigcap_{n\geq 0,j\in J_{n}(k,\alpha)} E_{n,j}(k,\alpha).
\eeq

We note that, uniformly in $\alpha$ and $k$,  
$$
	\lfloor (2K)\ 2^{n/2}/(1+n^2)\rfloor \;\leq\; |J_{n}(k,\alpha)|\; \leq \;  (2K)\ 2^{n/2}.
$$ 
Hence, by Remark~\ref{rem: extends proposition}, 
\beq
	\label{eq:controle G}
	\bP\left(\overline{G(k,\alpha)}\right) \; \leq \; \sum_{n\geq 0}(2K) 2^{n/2} \frac1C \ed^{-C (\alpha+n^2)} 
\eeq
and, by Borel Cantelli lemma, we deduce that
\beq
	\label{eq:boite cle}
	\text{almost surely for all } k\geq 1 \text{ there exists } \alpha_k\geq 1 \text{ so that } u\in G(k,\alpha_k).
\eeq
This last property we be useful in many of the following proofs. A first consequence is the following 
\begin{remark}\label{re:grille}
An easy consequence of \eqref{eq:boite cle} is that almost surely, for any $i\in \Z$, 
the set  
$P_i=\{y \in \R : \forall s \in [2^{i-1},2^{i}], u(s,y) > 0 \}$ 
is (infinite and) not bounded. 
Indeed almost surely, $u\in G(k,\alpha_{k})$ for all $k\geq i$ 
so that there are at least $|J_{k-i}(k,\alpha_k)|$ points in $P_i$ separated by a distance at least $\sqrt{2^i}$. 
As $|J_{k-i}(k,\alpha_k)|$ goes to infinity when $k\to \8$, this yields the result. 
The same result holds of course for the set 
$N_i=\{y \in \R :  \forall s \in [2^{i-1},2^{i}], u(s,y) < 0 \}$.
\end{remark}

\bigskip
\textbf{Expected size of $u$ and its derivatives.} 
We prove here some quantitative estimates on the field $u$ and its derivatives.
\begin{lemma}\label{lem: upper bounds on u and del u}
	For any $\varepsilon > 0$, there exists $C > 0$ such that 
	\begin{align}
		\bE\left(\sup \left\{t^{\frac14 + \varepsilon} u(t,x) \, : \,  t\in ]0,1], x \in [-1,1] \right\}\right) \; &\le \; C ,
		\label{eq: bound on sup u}\\
		\bE\left(\sup \left\{t^{\frac34 + \varepsilon} \partial_x u(t,x) \, : \,  t\in ]0,1], x \in [-1,1] \right\}\right) \; &\le \; C ,
		\label{eq: bound on sup del x u}\\
		\bE\left(\sup \left\{t^{\frac54 + \varepsilon} \partial_t u(t,x) \, : \,  t\in ]0,1], x \in [-1,1] \right\}\right) \; &\le \; C. 
		\label{eq: bound on sup del t u}
	\end{align}
	As the field $u\stackrel{(law)}{=}-u$, similar estimates hold for the infimum instead of the supremum.
\end{lemma}

\begin{remark}
We stress that the result is false if $\varepsilon =0$ as these supremum are infinite almost surely in this case. 
Indeed, let us consider for example the field $u$.  
First, for all $t>0$ and $x\in\R$, the variables $t^{1/4}u(t,x)$ are identically distributed joint Gaussian variables. 
Second, given any $n\ge 1$ and points $x_1,\dots ,x_n$ with $-1 \le x_1<\dots <x_n \le 1$, 
the Gaussian vector $(t^{1/2}u(t,x_1), \dots , t^{1/2}u(t,x_n))$ becomes uncorrelated as $t \to 0$.
From this, one concludes that $\sup_{t\in]0,1],x\in[-1,1]}t^{1/4}u(t,x) = + \infty$ almost surely. 
\end{remark}

\begin{proof}
Before starting the proof, we remind that we have already obtained the expression $\bE (u(t,x) u(s,y)) = P_{t+s}(x-y)$ in \eqref{eq: covariance}. 
Analogously, we derive
$$
	\partial_x u (t,x) = -\partial_x\int_\R  P_t (x-y) \dd B_y =- \int_\R P_t' (x-y) \dd B_y
$$ 
and thus
\begin{align}
	\bE (\partial_x u(t,x) \partial_x u(s,y))
	\; &= \; 
	\int_\R  P_t'(x-z)  P_s'(y-z) \dd z
	\; = \; 
	\partial_x \partial_y P_{t+s}(x-y) \nonumber\\
	\; &= \; 
	- \partial_{xx} P_{t+s}(x-y) ,
	\label{eq: covariance derivative}
\end{align}
as well as
$$
	\bE (\partial_t u(t,x) \partial_t u(s,y)) = \partial^4_{x} P_{t+s}(x-y).
$$

The Lemma follows from Dudley's theorem, see e.g.~\cite{ledoux talagrand}.
Let us show \eqref{eq: bound on sup u}.
We define a metric $d$ on $]0,1]\times[-1,1]$ by 
\begin{equation}\label{eq: dudley distance}
	d((t,x),(s,y)) \; = \; \left(\bE \left(t^{\frac14 + \varepsilon} u(t,x) - s^{\frac14 + \varepsilon} u(s,y)\right)^2 \right)^{\frac12}.
\end{equation}
Given $\delta > 0$, let $N(\delta)$ be the minimal number of balls of radius $\delta$ for the metric $d$ needed to cover $]0,1]\times[-1,1]$. 
Dudley's theorem asserts that 
\begin{equation}\label{eq: dudley bound}
	\bE\left(\sup \left\{t^{\frac14 + \varepsilon} u(t,x) \, : \,  t\in ]0,1], x \in [-1,1] \right\}\right) \; \le \; 24 \int_0^\infty (\ln N(\delta))^{\frac12} \dd \delta .
\end{equation}

Given $\delta > 0$, let us derive a bound on $N(\delta)$. 
From \eqref{eq: covariance} and \eqref{eq: dudley distance}, we compute 
\begin{equation}\label{eq: explicit dudley distance}
	d((t,x),(s,y))^2 
	\; = \;
	(4\pi)^{-\frac12}
	\left( 
	t^{2\varepsilon} + s^{2 \varepsilon} - 2 s^{\frac14 + \varepsilon}t^{\frac14 + \varepsilon}\left(\frac{s+t}2\right)^{-\frac12}\ed^{-\frac{(x-y)^2}{2(t+s)}}
	\right).
\end{equation}
Hence the bounds
\begin{equation}\label{eq: bound distance}
	d((t,x),(s,y)) 
	\; \le \;
	(4\pi)^{-\frac14} (t^{2\varepsilon} + s^{2 \varepsilon} )^{\frac12}  
	\; \le \; 
	(4\pi)^{-\frac14}\sqrt 2 .
\end{equation}
Therefore $N(\delta) = 1$ as soon as $\delta \ge (4\pi)^{-\frac14} \sqrt 2$.
Let us thus assume that $0 < \delta < (4\pi)^{-\frac14}\sqrt 2$. 
It follows from \eqref{eq: bound distance} that the set 
$$
	B_0 \; = \;  \left]0,c\,\delta^{\frac1\varepsilon}\right] \times [-1,1]
	\quad \text{with} \quad c \; = \; \pi^{\frac1{4\varepsilon}}
$$ 
is contained in single ball of radius $\delta$. 
On $(]0,1]\times[-1,1])\backslash B_0$, we will show that there exists $C>0$ such that 
\begin{equation}\label{eq: bound distance norm}
	d((t,x),(s,y)) \; \le \; C \delta^{-\frac1{2\varepsilon}} \left( |t-s| + |x-y| \right).
\end{equation}
This implies that there exists $C > 0$ such that $N(\delta) \le C \delta^{- 2 \left(1 + \frac1{2\varepsilon}\right)}$, 
hence that the integral in \eqref{eq: dudley bound} converges (to a value that depends on $\varepsilon$). 

Let us show \eqref{eq: bound distance norm}. By the triangle inequality,{ it holds that}
$$
	d((t,x),(s,y)) \; \le \; d((t,x),(t,y)) + d((t,y),(s,y)).
$$
First, from \eqref{eq: explicit dudley distance}, 
$$
	d((t,x),(t,y))^2
	\; = \;
	\pi^{-\frac12} t^{2\varepsilon} \left( 1 - \ed^{- \frac{(x-y)^2}{4t}} \right) 
	\; \le \; 
	C \frac{(x-y)^2}{t}
	\; \le \; 
	C \delta^{-\frac1\varepsilon} (x-y)^2
$$
where we have used the bounds $t\le 1$ and $1 - \ed^{-z}\le z$ for all $z\ge0$ to obtain the first inequality, 
and $t > c \delta^{\frac1\varepsilon}$ to get the second one. 
Next, from \eqref{eq: explicit dudley distance} again, 
\begin{align*}
	d((t,y),(s,y))^2
	\; &= \; 
	(4 \pi)^{-\frac12} t^{2 \varepsilon}
	\left(
	1 + \left(1 + \frac{s-t}{t} \right)^{2\varepsilon} - 2 \left(1 + \frac{s-t}{t} \right)^{\frac14 + \varepsilon} \left(1 + \frac{s-t}{2t}\right)^{-\frac12}
	\right) \\
	&=: \; t^{2\varepsilon}\varphi \left( \frac{s-t}{t} \right)
	\; \le \; 
	C \left( \frac{s-t}{t} \right)^2 
	\; \le \; 
	C \delta^{- \frac1\varepsilon} (t-s)^2
\end{align*}
where the first inequality follows from the fact that the distance $d$ is bounded by \eqref{eq: bound distance} and that $\varphi(0) = \varphi' (0) = 0$, 
and where the second one is obtained thanks to the bound $t > c \delta^{\frac1\varepsilon}$.

The proof of \eqref{eq: bound on sup del x u} is analogous and we only outline the main steps. 
This time, 
$$
	d((t,x),(s,y))^2 
	\; = \;
	(16\pi)^{-\frac12}
	\left( 
	t^{2\varepsilon} + s^{2 \varepsilon} - 2 s^{\frac34 + \varepsilon}t^{\frac34 + \varepsilon}\left(\frac{s+t}2\right)^{-\frac32}
	\left(1 - \frac{(x-y)^2}{t+s} \right)\ed^{-\frac{(x-y)^2}{2(t+s)}}
	\right).
$$
Since the function $z \mapsto (1 - z^2) \ed^{- z^2/2}$ is bounded, we obtain a bound analogous to \eqref{eq: bound distance}: 
$$
	d((t,x),(s,y)) 
	\; \le \;
	C (t^{2\varepsilon} + s^{2 \varepsilon} )^{\frac12}  
	\; \le \; 
	C .
$$
Hence, again, it is enough to show \eqref{eq: bound distance norm} for $t > c \delta^{\frac1\varepsilon}$ for some $c>0$, 
and the rest of the proof uses only completely similar computations. 

The proof of \eqref{eq: bound on sup del t u} is analogous.
\end{proof}


\section{Proof of Proposition~\ref{pro: X well defined} and Theorem~\ref{the: short time}}\label{sec: short time}
\begin{proof}[Proof of Proposition~\ref{pro: X well defined}]
Let us first show that, almost surely, there exists $\varepsilon > 0$ so that $(X_t)_{0 \le t\le \varepsilon}$ is defined as the fixed point of the map
$$
	\Phi \; :\;  \cC\left([0,\gep],[-1,1] \right)\to \cC\left([0,\gep],[-1,1] \right), 
	\quad
	f\mapsto  \Phi(f) = t\mapsto \int _0^t u(s,f_s) \dd s.
$$
First, if we choose $\gep$ small enough, $\Phi$ is well defined.
Indeed, thanks to Lemma \ref{lem: upper bounds on u and del u},
the time integral is a.s. convergent and moreover, taking for example $\delta = 1/10$,{ we find $C>0$ such that}, for all $f\in \cC\left([0,\gep],[-1,1] \right)$, 
$$
	\|\Phi(f)\|_{\8} \; \le \; C \int _0^\gep \frac{\dd s}{s^{1/4+\delta}}  \; \le \; 1
$$ 
if $\gep$ is chosen small enough.

Next, $\Phi$ is contracting if $\varepsilon$ is small enough. 
Indeed, there exists $C>0$ such that, for all $f,g \in \cC\left([0,\gep],[-1,1] \right)$,
$$
\ba
	\|f-g\|_{\8} &\;\le\; \int_0^{\gep} |u(s,f_s)-u(s,g_s)| \dd s\\
	&\;\le\; \|f-g\|_{\8} \int_0^{\gep} \sup_{x\in [-1,1]} |\partial_x u(s,x)| \dd s 
	\; \le \; C\int_0^{\gep} \frac{\dd s}{s^{3/4+\delta}} 
	\; \le \; 1
\ea
$$
if $\gep$ is chosen small enough.

It is thus almost surely possible to define $(X_t)_{0\leq t \leq \gep}$ as the unique fixed point of $\Phi$.
Clearly this process is continuous and satisfies \eqref{eq:X} for $0 \le t \le \varepsilon$.
Let us show that this process can be extended on $\R_+$.

We define $(X_t)_{t\in I}$ as the maximal solution for the Cauchy problem {$\partial_t X_t = u(t,X_t)$}
with the condition that, at time $t=\gep$, $X_t$ coincides with $X_\gep$ found above. 
We already know that $\inf I=0$ so that it remains to prove that $\sup I=+\8$. 
If $t_\star =\sup I<\8$ then, $(X_t)_{0<t < t_\star}$ explodes before time $t_\star$, i.e. $\lim_{t\to t_\star^-}|X_t|=+\8$. 
This is impossible thanks to Remark \ref{re:grille}. 
Indeed, let $i$ be the integer such that $2^{i-1}< t_\star \leq 2^i$ and choose $x\in P_i$ and $y\in N_i$ so that $x < X_{2^{i-1}} <y$. 
This implies that $X_s\in]x,y[$ for all $2^{i-1}\leq s < t_\star$ and this is a contradiction.
\end{proof}

\begin{proof}[Proof of Theorem~\ref{the: short time}]
We fix some $\theta_0>0$. For $0 < \theta \le \theta_0$ and $T > 0$, we decompose
$$
	\frac{X_{\theta T}}{T^{3/4}} 
	\; = \; 
	\frac1{T^{3/4}} \int_0^{\theta T} u (s,0) \dd s
	\; + \;  
	\frac1{T^{3/4}} \int_0^{\theta T} \big( u(s,X_s) - u (s,0) \big)\dd s .
$$
Thanks to the scaling relation~\eqref{eq: scaling}, 
$$
	\left(\frac1{T^{3/4}} \int_0^{\theta T} u (s,0) \dd s\right)_{0\le \theta \le \theta_0}
	\; = \; 
	\left(\int_0^\theta u (s,0) \dd s\right)_{0\leq \theta \le \theta_0}
	\quad \text{in law}. 
$$
Hence, it is enough to show that almost surely 
$$
	\sup_{0 < \theta \le \theta_0} \left|\frac1{T^{3/4}} \int_0^{\theta T}  u(s,X_s) - u (s,0)\ \dd s \right|\; \to \; 0 \quad \text{as} \quad T \to 0. 
$$

Let $\delta = 1/10$.
Thanks to Lemma \ref{lem: upper bounds on u and del u}, almost surely, there exists a constant $C>0$ so that for all $0 < t \le 1$ and all $x\in [-1,1]$,  
we have the bounds $|t^{\frac14 + \delta} u(t,x)|<C$ and $|t^{\frac34 + \delta} \partial_x u(t,x)|<C$.  
From now on, by continuity, we take $T$ small enough so that $\theta_0 T \leq 1$ and $\sup_{s\leq \theta_0 T}|X_s|\leq 1$.
It thus holds that for $0\leq t \leq \theta_0T$,
$$
	|X_t| \; \leq \; |\int_0^t u(s,X_s)\ \dd s| 
	\; \le \; \int_0^t \sup_{x\in [-1,1]} |u(s,x)|\ \dd s
	\; \le \;  C\int_0^t \frac{\dd s}{s^{1/4+\delta}} 
	\; \le \; 2C t^{3/4-\delta} . 
$$ 
Next, for all $0 \le \theta \le \theta_0$, 
$$
\ba
	\left|\frac1{T^{3/4}}\int_0^{\theta T} u(s,X_s)-u(s,0) \dd s \right|
	& \; \le \; \frac1{T^{3/4}} \int_0^{\theta T} |u(s,X_s)-u(s,0)| \dd s \\
	& \; \le \; \frac1{T^{3/4}} \int_0^{\theta T} \sup_{x\in [-1,1]} |\partial_x u(s,x)| |X_s| \dd s \\
	& \; \le \; \frac{2C^2}{T^{3/4}} \int_0^{\theta T} \frac{1}{s^{3/4 + \delta}}  s^{3/4-\delta} \dd s
	\; \le \; 4C^2 \frac{(\theta_0 T)^{1 - 2 \delta}}{T^{3/4}}
\ea
$$
and the last bound (uniform on $0\leq \theta \leq \theta_0$) converges to $0$ as $T \to 0$. 
\end{proof}


\section{Proof of Lemma~\ref{lem: definition r} and Proposition~\ref{pro: Z process}}\label{sec: zeros}

We start with a Lemma, that guarantees that the zeros of $u$ are almost surely never degenerate, 
i.e. either $\partial_x u$ or $\partial_t u$ is non-zero whenever $u$ vanishes.
This will enable us to invoke the implicit function theorem in several places. 
Moreover, we show also that there are only countably many isolated points where $\partial_x u$ vanishes, corresponding to the tops of the blue curves on Fig.~\ref{fig: 0LRX}.
This is the key ingredient to show item 3 in Proposition~\ref{pro: Z process}.
\begin{lemma}\label{lem: proba zero events}
	The field $u$ satisfies 
	\begin{enumerate}
	\item $\bP\big(\exists (t,x)\in \R_+^* \times \R : u(t,x) = \partial_t u (t,x) = \partial_x u (t,x) = 0\big) \; = \; 0.$
		\item Almost surely, on any compact set $K \subset \R_+^* \times \R$, the set of points where $u(t,x)=\partial_x u (t,x) = 0$ is finite.
	\end{enumerate}
\end{lemma}

\begin{remark}
As our proof shows, the first item holds actually for any smooth field $u$ such that $(u,\partial_t u , \partial_x u)$ has a locally bounded density around $0$. 
\end{remark}

\begin{proof}
	For the first item, let us consider the field $\varphi$ defined by
	\beq
		\varphi(t,x)\; = \; \big( u(t,x) , \partial_t u (t,x) , \partial_x u (t,x) \big), \quad t>0,x\in \R.
	\eeq
	From the scaling relation \eqref{eq: scaling}, we obtain also
	\beq\label{eq: second scaling relation}
		\big(\varphi(t,x)\big)_{t,x}
		\stackrel{(law)}{=}
		\big( \alpha^{1/4} u(\alpha t, \alpha^{1/2}x) ,  \alpha^{5/4} \partial_t u (\alpha t,\alpha^{1/2}x) ,  \alpha^{3/4}\partial_x u (\alpha t,\alpha^{1/2}x) \big)_{t,x}.
	\eeq
	By sigma additivity, together with the scaling relation~\eqref{eq: second scaling relation} and the space translation invariance
	we find that it suffices to show that 
	$$
		\bP \big(\exists (t,x) \in [1,2]\times [-1,1] : \varphi (t,x) = 0 \big)  \; = \;  0,
	$$
	
	To prove this, let us first show that there exists some $C > 0$ such that, for any $(t,x) \in [1,2]\times [-1,1]$ and for any $\varepsilon > 0$, 
	\begin{equation}\label{eq: cubic bound}
		\bP(|\varphi (t,x)| < \varepsilon) \; < \; C \varepsilon^3
	\end{equation}
	where $|\cdot|$ denotes the Euclidian norm.
	First, since $t \le 2$, {by the scaling relation \eqref{eq: second scaling relation} and translation invariance, }for all $(t,x) \in [1,2]\times [-1,1]$
	\begin{align*}
		\bP(|\varphi (t,x)| < \varepsilon) \leq \bP (|\varphi(2,0)|< \varepsilon).
	\end{align*}
	Therefore, to show \eqref{eq: cubic bound}, since $\varphi(2,0)$ is Gaussian, it suffices to show that its covariance is non-degenerate, i.e.\@ invertible. 
	Since 
	$$
		\bE( \varphi_i \varphi_j ) \; = \;  \int_\R \partial^i P_2 (z) \partial^j P_2 (z) \dd z 
	$$
	
	for $1 \le i,j\le 3$,
	with the notation $(\partial^1,\partial^2,\partial^3) = (1, \partial_t , \partial_x)$, 
	{and since $P_2(\cdot)$, $\partial_t P_2(\cdot)$ and $\partial_x P_2(\cdot)$ are linearly independent as elements of $L^2(\R)$, 
	the covariance is indeed non-degenerate. 
}
	
	Next, because $\varphi$ is smooth, it holds that 
	\begin{multline}\label{eq: bound with N}
		\{\exists (t,x)\in[1,2]\times[-1,1]:\varphi (t,x) = 0\}
		\; = \; \\
		\bigcup_{N \in \N^*}
		\left\{
		\exists(t,x)\in[1,2]\times [-1,1] : \varphi (t,x) = 0 \text{ and }  \|\varphi' \|_{\8} < N
		\right\}
	\end{multline}
	with $\|f\|_{\8} = \max_{(t,x) \in [1,2]\times [-1,1]}\|f(t,x)\|$ for any continuous function $f$ on $[1,2]\times [-1,1]$
	with values in linear application from $\R^2$ to $\R^3$, and where $\|\cdot\|$ is the operator norm when both spaces are endowed with the euclidian norms.
	Thus it suffices to show that, for any $N \in \N^*$, the probability of the corresponding set in the union in the right hand side of \eqref{eq: bound with N} is zero. 
	Let $N \in \N^*$, let $\varepsilon > 0$ and let us define the points 
	$$
		(t_i,x_j) \; = \; \left(1 + i\varepsilon, j\varepsilon \right), 
		\quad i \in \N, \; i < 1/\varepsilon,
		\quad j \in \Z, \; |j| < 1/\varepsilon.
	$$
	The number of such points is bounded by $2 / \varepsilon^2$.
	Now, under the condition $\|\varphi' \|_{\8} < N$,
	if $\varphi (t,x) = 0$ for some $(t,x)\in[1,2]\times [-1,1]$, then $|\varphi (t_i,x_j)| \le \sqrt 2 \varepsilon N$ for one of the points $(t_i,x_j)$ at least.
	Hence, using \eqref{eq: cubic bound}, we obtain
	\begin{align*}
		\bP \left(\exists(t,x)\in[1,2]\times [-1,1] : \varphi (t,x) = 0 \text{ and } \|\varphi' \|_{\8} < N\right)
		\; &\le \; \sum_{i,j} \bP \left(|\varphi (t_i,x_j)| \le \sqrt 2 \varepsilon N\right)\\ 
		\; &\le \; \frac{2C}{\varepsilon^2} \left(\sqrt 2 \varepsilon N\right)^3.
	\end{align*}
	Since $\varepsilon$ may be taken arbitrarily small for given $N\in\N^*$, the left hand side vanishes for any $N\in \N^*$.

	We turn to the second item. 
	As $K$ is compact it is enough to prove that the set of $(s,y)\in K$, so that $u(s,y) = \partial_x u (s,y) = 0$, 
	has only isolated points. Using item $1$, we may assume that almost surely for all points of this set $\partial_t u (s,y) \ne 0$. 
	We consider one of these points,
 	and using the implicit function theorem, we know that there exists a real function $S$ defined in a neighborhood of $y$ 
	so that the set of zeros of the field $u$ coincides with the graph of this function on a neighborhood of $(s,y)$.
	The function $S$ satisfies
	\begin{align*}
		S'(z) \; &= \; - \frac{\partial_x u (S(z),z)}{\partial_t u (S(z),z)} , \\
		S''(z) \; &= \; - 1 + S'(z) 
		\left( \frac{\partial_x u(S(z),z) \partial_{tt}u(S(z),z)}{(\partial_t u(S(z),z) )^2} - 2\frac{\partial_{tx}u(S(z),z)}{\partial_t u(S(z),z)} \right)	
	\end{align*}
	for $z$ in a neighborhood of $y$. Therefore, in a neighborhood of $y$, $\partial_x u (S(z),z) = 0$ if and only if $S'(z) = 0$ and, 
	since $S'(y) = 0$, we have $S''(z) < 0$, and thus also $S'(z) \ne 0$ for $z \ne y$.  
\end{proof}

We have now all ingredients for the 
\begin{proof}[Proof of Lemma \ref{lem: definition r}]
By Lemma~\ref{lem: proba zero events}, 
one may assume that almost surely on every point $(s,y)$ such that $u(s,y)=0$, either $\partial_t u (s,y) \ne 0$ or $\partial_x u (s,y) \ne 0$.
The main observation is that, if $(s,y)$ is such that $u(s,y) =0$ and  $\partial_x u (s,y) = 0$, then there exists $\varepsilon > 0$ such that 
\begin{equation}\label{eq: no zero below}
	u(s',y') \ne 0 \text{ for all } (s',y') \in ]s,s+\varepsilon[\times]y-\varepsilon,y+\varepsilon[.
\end{equation}
Indeed, as $\partial_t u (s,y) = \partial_{xx} u (s,y) \ne 0$, we may assume that $\partial_t u (s,y) > 0$ (the other case being analogous). 
By continuity, there exists $\varepsilon > 0$ such that $\partial_t u(s',y') > 0$ for all $(s',y') \in ]s,s+\varepsilon[\times]y-\varepsilon,y+\varepsilon[$, 
and $u(s,y')\ge 0$ for all $y' \in ]y-\varepsilon,y+\varepsilon[$. 
Therefore $u(s',y')>u(s,y')\ge 0$ or all $(s',y') \in ]s,s+\varepsilon[\times]y-\varepsilon,y+\varepsilon[$, which shows the claim.
Let now $t>0$ and $x \in \cZ_t^{\mathrm s} \cup \cZ_t^{\mathrm u}$. 
We consider the set 
$$
	\cF\; = \; \left\{
	\ba
	& 0\leq  s_0 < t : \textrm{there exists a function }\ r_{(t,x)}:\; ]s_0,t]\to \R, \textrm{ and} \\
	& \textrm{ a neighbourhood } \cV\ \textrm{ of } (t,x) \textrm{ containing the graph of } r_{(t,x)}, \\ 
	&\textrm{ such that} \textrm{ for all }(s,y)\in \cV,\ u(s,y)=0 \Leftrightarrow y= r_{(t,x)}(s) 
	\ea \right\}.
$$
Since $\partial_x u(t,x) \ne 0$, the implicit function theorem guarantees that $\cF \ne \emptyset$. Let 
$$
	s_{min} \; = \; \inf \cF
$$
and let us prove by contradiction that $s_{min} = 0$. Assume that $s_{min}>0$. 

Remind the definition of $G(k,\alpha_k)$ in \eqref{eq: def G} and also property \eqref{eq:boite cle}. 
We choose $k$ large enough so that $t_0(k)=2^k\geq t$ and $|x|\leq \frac{K}{3} \alpha_k \sqrt{t_0(k)}$. 
Since $u\in G(k,\alpha_k)$ and since the graph of $ r_{(t,x)} $ does not intersect $\{u>0\}\cup \{u<0\}$, 
we obtain that for all $i \geq 0$, $|r_{(t,x)}(t_i(k))-r_{(t,x)}(t_{i+1}(k))|\leq 2\ell_i(k)$. 
Hence, by definition of $K$ in \eqref{eq:defn0}, for all $s_{min}\leq s \leq t$, $|r_{(t,x)} (s)|\leq  \frac{2K}{3} \alpha_k \sqrt{t_0(k)}$, 
and the set $\{(s,r_x(s)):s_{min}<s\le t\}$ is bounded. By compactness, there exists $y\in\R$ such that $(s_{min},y)$ lies in its closure. 
Therefore, by continuity, $u(s_{min},y)=0$ and, by the implicit function theorem again, $\partial_x u (s_{min},y) = 0$ as otherwise one should have $\inf\cF < s_{min}$. 
We reach a contradiction with \eqref{eq: no zero below}.

Let us next show that $r_{(t,x)}$ is continuous in $0$. This is a consequence of the fact that $r_{(t,x)}$ satisfies Cauchy property as $s$ goes to $0$. 
Indeed, with the same argument as above, for all $j$ large enough and $0<s, s'<t_j(k)$,
$$
	|r_{(t,x)}(s) - r_{(t,x)}(s')| \; \le \; \sum_{n\geq j} 2 \ell_n(k),
$$
and this last sum goes to zero as $j\to+\8$.

Finally, that stable zeros remain stable, and unstable ones remain unstable, follows from the fact that $\partial_x u (s,r_{(t,x)}(s))\ne 0$ for all $s\in ]0,t]$, 
as the above argument shows. 
The expression \eqref{eq:deriv r} follows from the implicit function theorem. 
\end{proof}

\begin{proof}[Proof of Proposition~\ref{pro: Z process}: 
existence of the processes $L$ and $R$ satisfying \eqref{eq: def Lt Rt} and \eqref{eq: alternative stable unstable}]
Let us first show that there exist processes $L$ and $R$ satisfying \eqref{eq: def Lt Rt} almost surely for any $t>0$. 
To fix the ideas, let us deal with $L$. 
As, almost surely, $u(t,\cdot)$ is analytic for any $t>0$, $\cZ_t$ has no accumulation points, 
and it is enough to prove that the set $\{x \in \cZ_t^{\mathrm s} \cup \cZ_t^{\mathrm u} : r_{(t,x)}(0) < 0\}$ is non-empty and bounded above.
Let us fix $t>0$ and show that, almost surely, for any $t' \in ]0,t]$, 
$\{x \in \cZ_{t'}^{\mathrm s} \cup \cZ_{t'}^{\mathrm u} : r_{(t',x)}(0) < 0\}$ is non-empty and bounded above.

We choose $k$ large enough so that $2^{k}\geq t$. Using Remark \ref{re:grille}, almost surely, 
$u(t',\cdot)$ changes sign infinitely many often on $]-\8,-K \alpha_k\ \sqrt{2^{k}}[$. 
As $u$ is continuous, each interval where $u$ changes sign intersects  $\cZ_{t'}$. 
One can actually say more as, from the first item in Lemma~\ref{lem: proba zero events}, 
we may assume that $\partial_{xx} u = \partial_t u \ne 0$ whenever $u = \partial_x u  = 0$ and thus, 
for all $y\in \cZ_{t'}^{\mathrm n}$, the function $u(t,\cdot)$ vanishes but does not change sign in a neighborhood of $y$. 
From this one {deduces} that each interval where $u$ changes sign intersects $\cZ_{t'}^{\mathrm s} \cup \cZ_{t'}^{\mathrm u}$ 
(we will use repeatedly this argument in the following).

Thus there exists $x\in   ]-\8,-K \alpha_k\ \sqrt{2^{k}}[\cap (\cZ_{t'}^{\mathrm s} \cup \cZ_{t'}^{\mathrm u})$. 
Arguing as in the proof of Lemma~\ref{lem: definition r}, 
we obtain that for all $0<s\leq t'$, $r_{t',x}(s)\leq -\frac{2K}3 \alpha_k\sqrt{2^{k}}$ and in particular $r_{t',x}(0)<0$. 
This implies that $\{x \in \cZ_{t'}^{\mathrm s} \cup \cZ_{t'}^{\mathrm u} : r_{(t',x)}(0) < 0\}$ is non-empty.
Moreover it is also bounded above as, with the same argument, for $x\in \cZ_{t'}^{\mathrm s} \cup \cZ_{t'}^{\mathrm u} $ such that $x\geq K \alpha_k \ \sqrt{2^{k}}$, 
it holds that $r_{t',x}(s)\geq \frac {2K}{3}\alpha_k\sqrt{2^{k}}>0$.

Second, let us show \eqref{eq: alternative stable unstable}. 
For this, let us first prove that the probability of the event
\begin{equation}\label{eq: the set W of proba 0}
	\cW \; = \; \{ \exists (t,x)\in \R_+^* \times \R : r_{(t,x)}(0) = 0\}
\end{equation}
vanishes. Let us decompose this event as
$$
	\cW \; = \; \bigcup_{t>0} \cW_t(0) \quad \text{with} \quad \cW_t(y) \; = \; \{\exists x \in \R : r_{(t,x)}(0) = y\}, \; y \in \R.
$$
Since, by Lemma \ref{lem: definition r}, the events $\cW_t(0)$ increase as $t$ decreases, it is enough to show that $\bP(\cW_t(0)) = 0$ for any $t>0$. 
Let $t>0$. As argued above, the set $\cZ_t^{\mathrm s} \cup  \cZ_t^{\mathrm u}$ is almost surely unbounded above and below, and countable.  
Let us denote its elements by $(z_k)_{k \in \Z}$, with $z_k < z_{k+1}$ for all $k\in \Z$ and $z_0 = \min (\cZ_t^{\mathrm s} \cup  \cZ_t^{\mathrm u})\cap \R_+$.
Therefore, given $y\in\R$, it holds that $\bP(\cW_t(y))>0$ if and only if $\bP(r_{(t,z_k)}(0) = y)>0$ for some $k\in\Z$. 
Since the atoms of a random variable are at most countable, the set of $y\in\R$ such that $\bP(\cW_t(y))>0$ is at most countable. 
As $\bP(\cW_t(y))$ is constant in $y\in\R$ by translation invariance, we deduce that $\bP(\cW_t(y)) = 0$ for all $y\in \R$. 

{On $\cW^c$,} let us assume by contradiction that there exists some $t>0$ such that $L_t,R_t \in \cZ_t^{\mathrm s}$
(one rules out analogously the case $L_t,R_t \in \cZ_t^{\mathrm u}$).
Since $u(t,x)<0$ for $x>L_t$ in a neighborhood of $L_t$ and since $u(t,x)>0$ for $x<R_t$ in a neighborhood of $R_t$, 
we find that there exists $x\in Z_t^{\mathrm u} \cap ]L_t,R_t[$.
By \eqref{eq: def Lt Rt}, we would have $r_{(t,x)}(0)=0$, but this is impossible if $\cW^c$ is realized. 
\end{proof}

\begin{proof}[Proof of item 1 in Proposition~\ref{pro: Z process}]
First, let us show that $Z$ is continuous in $t=0$. 
This follows from the fact that $L$ and $R$ are continuous in $t=0$. 
To fix the ideas, let us show this for $L$. 
We actually prove a bit more: For all $\gep >0$, almost surely if $t>0$ is small enough, $|L_t| \le  t^{1/2-\gep}$.
Let $\varepsilon > 0$.
As, almost surely, $u\in G(0,\alpha_{0})$ and $r_{t,L_t}(0)<0$, arguing as in the proof of Lemma \ref{lem: definition r}, 
we obtain that there exists $C > 0$ such that, for all $0<t\leq 1$,
$$
	L_t 
	\; \le \;  
	\sum_{n\geq \lfloor -\log_2 t \rfloor} 2 \ell_n(0 ,\alpha_0)
	\; = \; 
	\sum_{n\geq \lfloor -\log_2 t \rfloor} 2 (\alpha_0+n^2) \sqrt{2^{-n}}
	\; \le \; 
	C (\log_2 t)^2 t^{1/2}.
$$
Thus, for $t>0$ small enough, the upper bound $L_t \le t^{1/2 - \varepsilon}$ holds.  
Moreover, the above bound implies that for $t>0$ small enough, 
$$
	t^{1/2-\gep}-\sum_{n\geq \lfloor -\log_2 t \rfloor} 2 \ell_n(0 ,\alpha_0) 
	\; \geq \;  
	2  \ell_{\lfloor -\log_2 t \rfloor}(0 ,\alpha_0).
$$
Therefore, the function $u(t,\cdot)$ changes sign in $[-t^{1/2-\gep},-\sum_{n\geq \lfloor -\log_2 t \rfloor } 2 \ell_n(0,\alpha_0)[$ 
so that this interval intersects $\cZ_t$. 
Using the same argument as in the proof of Lemma \ref{lem: definition r}, 
one can even says that this interval intersects $\cZ_t^{\mathrm s}\cup \cZ_t^{\mathrm u}$ and we consider some $x$ in this intersection. 
As $u\in G(0,\alpha_{0})$, $r_{t,x}(0)\leq x+\sum_{n\geq \lfloor -\log_2 t \rfloor} 2 \ell_n(0) <0$ and this implies that $L_t\geq x \geq -t^{1/2-\gep}$.

Second, let $t>0$ and let us prove that $Z$ is càdlàg at $t$. 
Because $L_t,R_t \in \cZ_t^{\mathrm s} \cup \cZ_t^{\mathrm u}$, the implicit function theorem implies that
there exist $\varepsilon > 0$ as well as $x_L,x_R \in \R$ so that 
$$
	L_t \; = \; r_{(t+\varepsilon,x_L)} (t), \qquad 
	R_t \; = \; r_{(t+\varepsilon,x_R)} (t).
$$
By definition of $L_t$ and $R_t$, it holds that $]L_{t},R_{t}[ \cap (\cZ_t^{\mathrm{s}} \cup \cZ_t^{\mathrm{u}}) = \emptyset$ 
so that the only zeros that could be in $]L_{t},R_{t}[$ are neutral, and there are only a finite number of them since $\cZ_t$ has no accumulation point. 
We call them $z_i$, $i=1,\cdots,n$ (of course $n$ can be $0$ and, even if we did not need to prove it for our purposes, we believe that $n$ is at most $1$). 
We claim that, for $\varepsilon > 0$ small enough, there is no zero of $u$ in the set 
$$
	\bigcup_{t<s\leq t+\varepsilon}]r_{(t+\varepsilon,x_L)} (s) , r_{(t+\varepsilon,x_R)} (s)[
$$
Indeed, otherwise, as $u$ is continuous there would be a sequence of zeros in this set converging to some $z_i$, and this is impossible due to \eqref{eq: no zero below}, 
or to $L_t$ or $R_t$ and this is also impossible thanks to the implicit function theorem as both points are in $\cZ_t^{\mathrm{s}} \cup \cZ_t^{\mathrm{u}}$. 
This implies that $Z_s = r_{(t+\varepsilon,x_L)} (s)$ for all $s \in [t,t+\varepsilon]$ 
or $Z_s = r_{(t+\varepsilon,x_R)} (s)$ for all $s \in [t,t+\varepsilon]$, and this proves thus that $Z$ is right continuous at $t$. 

If $n\ge 1$ and if $i \in \{1,\dots , n\}$,
let $S^i$ be the function defined as in the proof of the second item of Lemma~\ref{lem: proba zero events}, here in the neighborhood of $z_i$.
From the properties of $S^i$, one deduces that,
for $\varepsilon > 0$ small enough, there exist two continuous functions $x^i_1$ and $x^i_2$, 
such that for all $t-\varepsilon \leq s < t$, $x^i_{1/2}(s)$ is in $\cZ_s^{\mathrm{s}} \cup \cZ_s^{\mathrm{u}}$, 
$x^i_1(s)<x <x^i_2(s)$ and the graphs of $x^i_{1/2}$ coincide with the zeros of $u$ in a neighborhood of $(t,z^i)$. 

Finally, we find that for all $t-\varepsilon \leq s < t$, 
\begin{multline*}
[r_{(t+\varepsilon,x_L)} (s), r_{(t+\varepsilon,x_R)}]\cap (\cZ_s^{\mathrm{s}} \cup \cZ_s^{\mathrm{u}}) \\
\; = \;  \{x^i_j(s),\ i=1,\dots,n ;\ j=1,2\}\cup \{r_{(t+\varepsilon,x_L)} (s), r_{(t+\varepsilon,x_R)}(s)\}
\end{multline*}
with the convention that the first set in the union in the right hand side is empty if $n=0$.
This implies that $Z$ coincides on$[t-\varepsilon,t[$ with one of these $(n+2)$ functions and thus that it is left continuous at $t$.
\end{proof}

\begin{proof}[Proof of item 2 in Proposition~\ref{pro: Z process}]
Suppose first that $Z_{t^-}\in\cZ_t^{\mathrm n}$. Using \eqref{eq: no zero below}, 
there exists $\varepsilon>0$ so that there is no zero in $]t,t+\varepsilon[\times]Z_{t^-}-\varepsilon,Z_{t^-}+\varepsilon[$. This implies that $Z$ is discontinuous at $t$. 
Suppose next that $Z_{t^-} \in \cZ_t^{\mathrm s} \cup \cZ_t^{\mathrm u}$, and thus by continuity that $Z_{t^-}\in\cZ_t^{\mathrm s}$. 
Without loss of generality, we may assume that $r_{({t,}Z_{t^-})}(0)<0$. 
First, if $z\in \cZ_t^{\mathrm s}$ satisfies $z < Z_{t^-}$, then $r_{(t,z)}(0) < r_{(t,Z_{t^-})}<0$ and thus $z \ne Z_t$. 
Second, if $z\in \cZ_t^{\mathrm s}$ satisfies $z > Z_{t^-}$, then there exists $z_0 \in ]Z_{t^-},z[\cap \cZ^{\mathrm u}_t$ and, 
by continuity, there exists $s\in]0,t[$ such that $Z_s < r_{(t,z_0)}(s)$. 
Therefore, $r_{(t,z_0)}(0) > 0$ and since $r_{(t,z)}(0) > r_{(t,z_0)}(0)$, this implies also $z\ne Z_t$. 
We conclude that $Z_t = Z_{t^-}$.
\end{proof}

\begin{proof}[Proof of item 3 in Proposition~\ref{pro: Z process}]
This follows from the second item in Lemma~\ref{lem: proba zero events} and the fact that $Z$ is discontinuous at $t$ if and only if $Z_{t-} \in \mathcal Z^{\mathrm n}_t$. 
\end{proof}

\begin{proof}[Proof of item 4 in Proposition~\ref{pro: Z process}]
In this proof, it is convenient to write $Z$ and $r$ as function of the environment. 
We fix $T>0$ and define $(\tilde Z_\theta^T)_{\theta\ge 0} = (T^{-1/2}Z_{\theta T})_{\theta\ge 0}$. 
Given an environment $u$, we also define $(u_T(t,x))_{t,x} = (T^{1/4} u (Tt,T^{1/2}x))_{t,x}$. 
Our goal is to prove that $\tilde Z^T (u) = Z(u_T)$. Hence, since $u$ and $u_T$ have the same law, this will imply our claim. 
Let $\theta>0$ and observe that
\begin{enumerate}
\item a real $x$ belongs to $\cZ_{\theta}(u_T)$ if and only if $T^{1/2}x \in \cZ_{\theta T}(u)$,
\item in this case both zeros are of the same type and, if moreover $x$ is not neutral, then for all $0\leq s \leq \theta$
$$
	 r_{(\theta,x)}(u_T)(s) \; = \; T^{-1/2} r_{(\theta T, T^{1/2}x)}(u)(sT). 
$$
To prove this last point we observe that the function $\phi:s\to T^{-1/2} r_{(\theta T, T^{1/2}x)}(u)(sT)$ is continuous, satisfies $\phi(\theta)=x$ and 
$u_T(s,\phi(s))=0$ for all $0<s\leq \theta$. This is enough to conclude as, by definition, $r_{(\theta,x)}(u_T)$ is the only function to have these properties.
\end{enumerate}
By definition of the process $Z$, these two points imply that $\tilde Z^T (u) = Z(u_T)$.
\end{proof}

\begin{proof}[Proof of item 5 in Proposition~\ref{pro: Z process}]
Let us first show that there exists $c>0$ so that $\bP(|Z_1|\ge z) \ge c \ed^{-z/c}$ for all $z\ge 0$. 
Given $z\ge 0$, the bound
$$
	\bP (|Z_1| \ge z)
	\; \ge \; 
	\bP (u (1,x) > 0 ,\  \forall x\in[-z,z])
$$
holds. 
Since $u$ is continuous almost surely, for any $x\in\R$, 
$$
	\{ u(1,x) > 0 \} \; = \; \bigcup_{a>0} \{u(1,y) > 0,\ \forall y \in[x-a,x+a]\}.
$$
Therefore, since $\bP(u(1,x)>0) = 1/2$, there exists $a>0$ such that,
$$
	\bP (u(1,y) > 0,\ \forall y \in[x-a,x+a]) \; \ge \; 1/4.
$$
Hence, assuming $z\ge a$, by translation invariance and since the variables $(u(1,x))_{x\in\R}$ are positively correlated, we obtain 
$$
	\bP (u (1,x) > 0,\ \forall x\in[-z,z])
	\; \ge \; 
	\big( \bP (u (1,y) > 0,\ \forall y \in [- a,+ a] ) \big)^{\lceil z/a \rceil}
	\; \ge \; 
	\ed^{-\ln 4 \lceil z/a \rceil  }.
$$

Second, let us show that there exists $c>0$ so that $\bP(|Z_1|\ge z)\le \ed^{-cz}/c$. 
We first remind that, from \eqref{eq:controle G}, there exists $C>0$ such that for all $\alpha \ge 1$,
\beq
	\label{eq:controle proba G}
	\bP(\overline{G(0,\alpha)})\leq \frac1C \ed^{-C \alpha}.
\eeq
If $u\in G(0,\alpha)$ the function $u(1,\cdot)$ changes sign on $]\frac 23 K\alpha,K\alpha[ $ so that, using the same argument as in the proof of Lemma \ref{lem: definition r}, this interval intersects $\cZ_1^{\mathrm u}\cup \cZ_1^{\mathrm s}$. We consider a point $x$ in this intersection. As $u\in G(0,\alpha)$, $r_{1,x}(0)\geq \frac 13 K\alpha>0$. With the same argument there exists $y\in ]-K \alpha,-\frac 23 K\alpha[ \cap ( \cZ_1^{\mathrm u}\cup \cZ_1^{\mathrm s})$ such that  $r_{1,y}(0)\leq -\frac 13 K\alpha<0$. This implies that $G(0,\alpha) \subset \{|Z_1| \leq \frac 23 K\alpha \}$ and, together with \eqref{eq:controle proba G}, concludes the proof of this point.

Let us finally show that $Z_1$ has a bounded density. 
For this, it is enough to show that the cumulative distribution function of $Z_1$ is Lipschitz. 
Let us thus show that there exists $C >0$ such that, for any $\varepsilon > 0$ and for any $x\in \R$, 
$$
	\bP (Z_1 \in [x,x+\varepsilon]) \; \le \;  C \varepsilon .
$$
We start with the bound 
$$
	\bP (Z_1 \in [x,x+\varepsilon])
	\; \le \; 
	\bP([x,x+\varepsilon] \cap \cZ_1 \ne \emptyset)
	\; = \; 
	\bP([0,\varepsilon]\cap \cZ_1 \ne \emptyset). 
$$
In the sequel, to simplify writings, let us write $u(x)$ for $u(1,x)$ for any $x\in\R$.
By a second order Taylor expansion, there exists a function $\theta : [0,\varepsilon] \to [0,\varepsilon]$ such that, for all $x\in [0,\varepsilon]$, 
\begin{equation}\label{eq: second order expansion}
	u(x) = u(0) + \partial_x u (0) x + \frac{\partial_{xx}u(\theta(x))x^2}{2}.  
\end{equation}
Let $\delta > 0$ to be fixed later and let us decompose $\bP([0,\varepsilon]\cap \cZ_1 \ne \emptyset)$ according to the following alternative: 
\begin{align}
	\bP([0,\varepsilon]\cap \cZ_1 \ne \emptyset)
	\; \le& 
	\; \bP(\exists y \in[0,\varepsilon] : |u(0) + \partial_x u (0) y| \le \delta ) \nonumber\\
	\; &+ \; \bP(\exists x \in [0,\varepsilon] : u(x)=0 \quad \text{and} \quad  \forall y \in [0,\varepsilon] : |u(0) + \partial_x u (0) y| > \delta ).
	\label{eq: decomposition alternative}
\end{align}
To get a bound on the first term, we notice that $u(0)$ and $\partial_x u (0)$ are independent Gaussian variables, 
and one finds that there exists $C>0$ such that, for any $\delta \in]0,\varepsilon]$ and any $\varepsilon>0$,
\begin{equation}\label{eq: first bound explicit}
	\bP(\exists y \in[0,\varepsilon] : |u(0) + \partial_x u (0) y| \le \delta ) \; \le \; C \varepsilon . 
\end{equation}
To get a bound on the second term, we use the expansion~\eqref{eq: second order expansion}: 
\begin{align}
	&\bP(\exists x \in [0,\varepsilon] : u(x)=0 \quad \text{and} \quad  \forall y \in [0,\varepsilon] : |u(0) + \partial_x u (0) y| > \delta ) \nonumber\\
	&= \; 
	\bP\left( \exists x \in [0,\varepsilon] : \frac{\partial_{xx}u(\theta(x))x^2}{2} = - (u(0) + \partial_x u (0) x)
	\quad \text{and} \quad  
	\inf_{y \in [0,\varepsilon]}  |u(0) + \partial_x u (0) y| > \delta 
	 \right) \nonumber\\
	 & \le \; 
	 \bP(\exists x \in [0,\varepsilon] : |\partial_{xx} u (\theta (x))|x^2 \ge 2\delta) \nonumber\\
	 & \le \; 
	 \bP \left( \sup_{x \in [0,\varepsilon]} |\partial_{xx} u (x)| \ge \frac{2\delta}{\varepsilon^2}\right) 
	 \; \le \; 
	 \frac{\varepsilon^2}{2 \delta} \bE \left(\sup_{x \in [0,\varepsilon]} \partial_{xx} u (x)\right)
	 \; \le \; 
	 \frac{C \varepsilon^2}{2 \delta}
	 \nonumber
\end{align}
where the last bound follows from Lemma~\ref{lem: upper bounds on u and del u}.
Therefore, taking $\delta = \varepsilon$, 
we obtain the claim by inserting this last bound together with \eqref{eq: first bound explicit} into \eqref{eq: decomposition alternative}. 
\end{proof}


\section{Proof of Theorem~\ref{the: long time}}\label{sec: long time}

Given $T>0$, let us define the processes $(\tilde Z_\theta^T)_{\theta \ge 0} = (T^{-1/2}Z_{\theta T})_{\theta \ge 0}$ as well as 
$(\tilde X_\theta^T)_{\theta \ge 0} = (T^{-1/2}X_{\theta T})_{\theta \ge 0}$.
Let us also define $(Y_\theta^T)_{\theta\ge 0}$ by $Y_0^T = 0$ and, for $\theta > 0$, 
$$
	\frac{\dd Y_\theta^T}{\dd \theta} \; = \; T^{1/4} u (\theta,Y_\theta^T).
$$ 
Note that this definition makes sense, as can be shown exactly with the same arguments as in the proof of Proposition~\ref{pro: X well defined}.

Let us first show that 
\begin{equation}\label{eq: egalite en loi entre couples}
	\big( \tilde Z_\theta^T, \tilde X_\theta^T\big)_{\theta \ge 0} 
	\; = \; 
	\big( Z_\theta,Y_\theta^T  \big)_{\theta \ge 0} 
	\quad \text{in law}.
\end{equation}
For this, it is convenient to explicitly write the couple of processes as a function of the environment. 
Given an environment $u$, let $(u_T(t,x))_{t\geq 0,x\in \R} = (T^{1/4} u (Tt,T^{1/2}x))_{t\geq 0,x\in \R}$ and let us show that 
\begin{equation*}
	\big( \tilde Z_\theta^T, \tilde X_\theta^T\big) (u) \; = \; \big( Z_\theta,Y_\theta^T  \big) (u_T)
\end{equation*}
for any $\theta \ge 0$. 
As $u$ and $u_T$ have the same law by the scaling relation~\eqref{eq: scaling}, this will imply \eqref{eq: egalite en loi entre couples}. 
The relation $\tilde Z_{\theta}^T (u) = Z_{\theta}(u_T)$ has already been shown in the proof of item 4 in Proposition~\ref{pro: Z process}.
To show $\tilde X_{\theta}^T (u) = Y_\theta^T (u_T)$, we notice that $\tilde X_{0}^T = 0$ and that for all $\theta > 0$,
$$
\frac{\dd \tilde X_\theta^T}{\dd \theta} \; = \; T^{1/4} u_T(\theta,\tilde X_\theta^T) 
$$
and the claim follows from the fact that these relations characterize the process $(Y_\theta^T)_{\theta \geq 0}(u_T)$. 

To prove Theorem~\ref{the: long time}, 
it is thus enough to prove that, almost surely, $Y^T$ converges to $Z$ in the $\cM_1$ topology on compact sets as $T\to\infty$. 
Indeed, this implies that $Y^T - Z$ converges to 0 in probability as $T \to \infty$ 
and, thanks to \eqref{eq: egalite en loi entre couples},
this implies that $(T^{-1/2}(X_{\theta T} - Z_{\theta T}))_{\theta \ge 0}$ converges to 0 in probability as $T\to\infty$.
For notational convenience, we will show that $(Y_t^T)_{t\in[0,1]}$ converges to $(Z_t)_{t\in[0,1]}$, 
but our proof still holds for $[0,1]$ replaced by any compact interval.

We use characterization $(v)$ of \cite{whitt} for the convergence in the $\cM_1$ topology. We first introduce some notations needed to state it. 
Given $a,b,c\in\R$, let 
$$
	\| a - [b,c] \| \; = \; \min_{\tau \in [0,1]} |a - (\tau b  + (1-\tau)c) |.
$$
For $\delta>0$ and $f,g \in D([0,1],\R)$, let 
$$
	v(f,g,t,\delta)= \sup\{ |f(t_1)-g(t_2)|,\ 0\vee (t-\gd) \leq t_1,t_2 \leq 1 \wedge (t+\gd) \}
$$
and 
$$
	w_s(f,t,\delta)= \sup\{\|f(t_2) - [f(t_1),f(t_3)]\|,\ 0\vee (t-\gd) \leq t_1< t_2 < t_3 \leq 1 \wedge (t+\gd) \}.
$$
The characterization is the following: $f^T\to f$ converges to $f$ as $T\to \infty$ for the $\cM_1$ topology on $\cD([0,1])$ if and only if 
\begin{enumerate}
\item $f^T(1)$ converges to $f(1)$
\item For all $0\leq t\leq 1$ that is not a discontinuity point of $f$.
$$
	\lim_{\delta\to 0}\lim_{T\to +\8} v(f^T,f,t,\delta)\; =\; 0.
$$
\item For all $0\leq t\leq 1$ that is a discontinuity point of $f$
$$
	\lim_{\delta\to 0}\lim_{T\to +\8} w_s(f^T,t,\delta)\; =\; 0.
$$
\end{enumerate}

The first point is actually a consequence of the second one as, for all $t \geq 0$ (and in particular for $t=1$), almost surely, $Z$ is continuous at $t$. Indeed we first observe, from item $4$ in Proposition \ref{pro: Z process}, that $s\to \bP(Z_{s^-}\neq Z_{s})$ is constant on $\R_+^*$. Then 
$$
	\int \bP(Z_{s^-}\neq Z_{s})\ ds \; = \; \bE\left(\int 1_{Z_{s^-}\neq Z_{s}}\ ds\right)\; = \; 0,
$$
as almost surely the discontinuity points of $Z$ are countable. This implies that $\bP(Z_{s^-}\neq Z_{s})=0$ for all $s>0$. 

We start with the proof of the second item in the above characterization: 
\begin{lemma}\label{lem: point2}
	Almost surely, for all $t_0 \in [0,1]$ such that $Z$ is continuous in $t_0$, 
	$$
		\lim_{\delta\to 0}\lim_{T\to +\8} v(Y^T,Z,t_0,\delta)=0.
	$$
\end{lemma}

\begin{proof}
We first consider the case $t_0=0$. 
As almost surely $u\in G(0,\alpha_0)$, for all $T>0$ and all $0 < t \leq 1$ small enough,
$$
	|Y_t^T| \leq   \sum_{n\geq \lfloor -\ln_2 t \rfloor} 2 \ell_n(0{,\alpha_0})= \sum_{n\geq  \lfloor -\ln_2 t \rfloor} 2 (\alpha_0+n^2) \sqrt{2^{-n}}.
$$
This implies that, for any $\varepsilon > 0$, $t>0$ small enough and for any $T>0$,
\begin{equation}\label{eq: a priori bound Y}
		|Y_t^T| \; \le \; t^{\frac12 - \varepsilon}.
\end{equation}
As $Z$ is right continuous at $0$ with limit $0$ this gives the result in the case $t_0=0$.

We consider $0<t_0 \leq 1$ so that $Z$ is continuous at $t_0$ and fix some $\gep>0$. 
To fix ideas, and as the other case is completely similar, let us assume that $Z_{t_0}=L_{t_0}$. 
Using Proposition \ref{pro: Z process}, there exists $\gd>0$ so that $Z$ is continuous on $[t_0-2\gd,t_0+\gd]$ so that, if $\gd>0$ has been chosen small enough, 
$$
	v(Z,Z,t_0,\delta)\; < \; \gep,
$$
and we only have to show that for $\gd >0$ small enough and all $T$ larger than some $T_0(\gd)$,
\beq\label{eq: to prove in proof of point 2 of criterion}
	\sup\{ |Z_s-Y^T_s|,\  t_0-\gd \leq s \leq  t_0+\gd \}\; <\; \gep.
\eeq

We use the notations $t_i=t_0+i \gd$, $i \in\{-2,-1,0,1\}$. 
We stress that for all $t \in [t_{-2},t_{1}]$, 
$Z_t = r_{(t_1,L_{t_1})}(t)$ as, for $x\in \cZ_{t_1}^{\mathrm s}$, 
$r_{(t_1,x)}$ is the only continuous function so that $r_{(t_1,x)}(t_1)=x$ and $u(s,r_{(t_1,x)}(s))=0$ for $0 < s \le t_1$. 
Using Lemma \ref{pro: Z process}, it is also possible to choose $\gd>0$ small enough so that the only neutral zeros of $u$ in 
\beq \label{eq:noneutral}
	\bigcup_{t_{-2}\leq s \leq t_1} [r_{(t_1,L_{t_1})}(s),r_{(t_1,R_{t_1})}(s)]
\eeq
lies in $\cZ_{t_0}$. Using Remark \ref{rq:aussiLR}, 
this choice for $\gd$ implies that $R$ is continuous on $[t_{-2},t_{-1}]$ and, with the same argument as above, that for all $t \in [t_{-2},t_{-1}]$, 
$R_t = r_{(t_{-1},R_{t_{-1}})}(t)$. 
Note however that it is not necessarily the case that $R_t = r_{(t_{1},R_{t_{1}})}(t)$, as $R$ could jump at time $t_0$.
	
{Before going to the proof of \eqref{eq: to prove in proof of point 2 of criterion} itself, let us first prove the following intermediate result: } For all $t>0$ so that $Z_t=L_t$, and if $\varepsilon > 0$ has been chosen small enough, there exists $T_0>0$ so that for all $T \ge T_0$ : 
\beq
	\label{eq: corridor 1}
	r_{(t,L_{t})} (s) - \varepsilon \; \le \; Y_s^T \; \le \; r_{(t,R_{t})} (s) - \varepsilon \quad \text{for all } s\in [0,t].
\eeq
	
	By the definition of $L$ and $R$ in Proposition~\ref{pro: Z process}, it holds that $r_{(t,L_{t})}(0)<0$ and $r_{(t,R_{t})}(0)>0$. 
	Since the functions $r_{(t,L_{t})}$ and  $r_{(t,R_{t})}$ are continuous, 
	and since $Y^T$ satisfies the bound \eqref{eq: a priori bound Y}, 
	we conclude that there exists $\tau \in ]0,t]$ so that \eqref{eq: corridor 1} holds for $s \in [0,\tau]$.
	Let us now assume that $s \in [\tau,t]$ and show the lower bound on $Y^T$ in \eqref{eq: corridor 1} (the proof of the upper bound is analogous). 
	By Lemma \ref{lem: definition r}, the function $s\to \partial_x u (s,r_{(t,L_{t})}(s))$ is continuous and strictly negative on $[\tau,t]$ so that by compactness, 
	there exists $c>0$ such that
	$$
		\partial_x u (s,r_{(t,L_{t})}(s)) \; \le \;  - c \quad \text{for all}\quad s \in [\tau,t].
	$$
	For $s\in [\tau,t]$, let
	$$
		\tilde r_{(t,L_{t})}(s) \; = \; r_{(t,L_{t})}(s) - \varepsilon. 
	$$
	A second order expansion yields
	$$
		u(s,\tilde r_{(t,L_{t})}(s))
		\; = \; 
		-\varepsilon \partial_x u (s,r_{(t,L_{t})}(s))
		+ \frac{\varepsilon^2}{2} \partial_{xx} u (s,y_\varepsilon(s))
	$$
	with $y_\varepsilon(s) \in [\tilde r_{(t,L_{t})}(t),r_{(t,L_{t})}(s)]$.
	By continuity of $\partial_{xx}u$ and compactness, there exists $K\ge 0$ such that 
	\beq
	\label{eq: lower bound u}
		u(s,\tilde r_{(t,L_{t})}(s)) 
		\; \ge \; 
		c \varepsilon - K \varepsilon^2 
		\; \ge \; 
		\frac{c \varepsilon}{2}
	\eeq
	for all $s\in [\tau, t]$, provided $\varepsilon > 0$ was taken small enough.
	Suppose now that the lower bound in \eqref{eq: corridor 1} is not satisfied 
	so that there exists $s\in[\tau,t]$ such that $Y_s = \tilde r_{(t,L_{t})}(s)$ and  
	$\partial_s Y_s \le \partial_s \tilde r_{(t,L_{t})}(s) = \partial_s r_{(t,L_{t})}(s)$ i.e.\@ explicitly
	\beq\label{eq:comparaison vitesse}
		T^{\frac14} u(s,\tilde r_{(t,L_{t})}(s)) \; \le \; - \frac{\partial_{xx} u (r_{(t,L_{t})}(s))}{\partial_x u (r_{(t,L_{t})}(s))}.
	\eeq
	Since the right hand side is uniformly bounded in $s\in[\tau,t]$, 
	the lower bound \eqref{eq: lower bound u} leads to a contradiction for $T$ large enough. This concludes the proof of \eqref{eq: corridor 1}.

	Let us now derive the result \eqref{eq: to prove in proof of point 2 of criterion} from \eqref{eq: corridor 1}. 
	We choose $T_0$ large enough so that \eqref{eq: corridor 1} holds both for time $t_1$ and $t_{-2}$.
	It remains to show that for $T$ large enough, $Y_{t}^T\le r_{(t_1,L_{t_1})} (t) + \varepsilon$ for all $t \in [t_{-1}, t_1]$. 
	For this, we first show that there exists $t_\star\in [t_{-2}, t_{-1}]$ such that $Y_{t_\star}^T\le r_{(t_1,L_{t_1})} (t_\star) + \varepsilon$.
	By the definition of $L$ and $R$, and since the set $\cW$ defined in \eqref{eq: the set W of proba 0} has probability $0$, for all $t>0$, it holds that
	$]L_t\cap R_t[\cap (\cZ_t^{\text{s}} \cup \cZ_t^{\text{u}})=\emptyset $. 
	Hence, the choice of $\gd$ made before \eqref{eq:noneutral} implies
	$$
		\left(\bigcup_{t_{-2}\leq s \leq t_{-1}} ]L_t,R_t[\right)\cap \{u=0\}\; = \; \emptyset .
	$$
	As $Z_t=L_t$ for all $t \in [t_{-2},t_{-1}]$, we obtain that $u (t,x) < 0$ for all $(t,x)\in \bigcup_{t_{-2}\leq s \leq t_{-1}} ]L_s,R_s[$ and thus, by compactness, 
	there exists $c > 0$ such that $u(t,x)<-c$ for all $(t,x)$ such that $L_t + \varepsilon \le x \le R_t - \varepsilon$ with $t\in[t_0 - 2 \delta, t_0 - \delta]$.
	Assume by contradiction that $Y_t > L_t + \varepsilon$ for all $t \in [t_{-2},t_{-1}]$. 
	Then, since we know that $Y_{t}\le R_t - \varepsilon$, we conclude that 
	$$
		Y_t  \le  Y_{t_{-2}} - c T^{1/4} (t - t_{-2}). 
	$$
	For $T$ large enough, this yields a contradiction.
	Second, once we know that $Y_{t_\star} \le r_{(t_1,L_{t_1})} (t_\star) + \varepsilon$, 
	we may proceed  as in the proof of \eqref{eq: corridor 1} and show that, for $T$ large enough, $Y_t \le r_{(t_1,L_{t})} (t) + \varepsilon$ for all $t \in [t_\star,t_1]$. 
\end{proof}

Next, we turn to the proof of the third item in the above characterization:

\begin{lemma}\label{lem: point3}
	Almost surely, for all $t_0 \in [0,1]$ such that $t_0$ is a jump point of $Z$, 
	$$
		\lim_{\delta\to 0}\lim_{T\to +\8} w_s(Y^T,t_0,\delta)=0.
	$$
	\end{lemma}
\begin{proof} 
We first describe how the environment looks like around a fixed jump point $t_0\in]0,1[$ of $Z$. 
In the following we will always suppose that $\gd>0$ is small enough so that $t_0$ is the only jump of $Z$ on $[t_0-\gd,t_0+\gd]$. 
As the three other cases are similar, we may also assume that $Z_s=L_s$ for all $t_0-\gd \leq s<t_0$ and $Z_s=R_s$ for all $t_0 \leq s \leq t_0+\gd$. 
We also consider $\gd$ small enough so that $R$ is continuous on $[t_0,t_0+\gd]$. 
Using Remark \ref{rq:aussiLR} this implies that  
\beq
\label{eq:sgnconstant}
\left(\bigcup_{t_0 < s \leq t_0+\gd} ]L_s,R_s[\;\right) \cap \{u=0\}\; = \; \emptyset.
\eeq

We first focus on the behaviour of the environment just before the jump and prove that, for all $\gep >0$ (small enough), 
there exists $\gd>0$ so that for all $t_0-\gd \leq t < t_0$,
\beq
\label{eq:before}
\ba
&u(t,Z_{t_0-}-\gep) > 0,\\
&u(t,Z_{t_0-}) < 0,\\
&Z_{t_0-\gd}\in ]Z_{t_0-}-\gep,Z_{t_0-}[.
\ea
\eeq
We next describe the environment just after the jump at time $t_0$: For all $\gep>0$ (small enough) there exists $\gd>0$ so that
\beq
\label{eq:after}
\ba
&L_t < Z_{t_0-}-\gep \quad \text{for all } t_0 \leq t \leq t_0+\gd,\\
& \sup\{|Z_s-Z_t|,\ t_0\leq s,t\leq t_0+\gd\} \leq  \gep.
\ea
\eeq

We delay the proof of these two points and first assume that \eqref{eq:before} and \eqref{eq:after} hold for some $\gep>0$ and $\gd>0$. 
We prove that it implies that, for $T$ large enough, 
\beq
\label{eq:Ysaut}
\ba
&Z_{t_0-}-\gep \leq  Y^T_t \leq Z_{t_0-} \quad \text{for all } t_0-\gd \leq t \leq t_0, \\
& Y^T \textrm{ is increasing on } [t_0,h] \textrm{ where } h=\inf\{t\geq t_0,\ Y^T_t \geq Z_{t}-\gep\} \wedge (t_0+\gd),\\
& Y^T_t \in [Z_{t}-\gep,Z_{t}+\gep] \quad \text{for all } h < t \leq t_0+\gd.\\
\ea
\eeq
Indeed, for the first point of \eqref{eq:Ysaut}, as $t_0-\gd$ is not a jump point of $Z$ and $Z_{t_0-\gd}\in ]Z_{t_0-}-\gep,Z_{t_0-}[$, Lemma \ref{lem: point3} ensures that for $T$ large enough $Y^T_{t_0-\gd}$ lies also in $]Z_{t_0-}-\gep,Z_{t_0-}[$ and both barriers defined in \eqref{eq:before} ensures that $Y^T$ stays in this interval till $t_0$.
For the second point of \eqref{eq:Ysaut}, from \eqref{eq:sgnconstant} and \eqref{eq:after} we deduce that $u>0$ in the domain $\{(t,x),t_0 \leq t \leq t_0+\gd; Z_{t_0-}-\gep \leq x \leq  R_{t}-\gep\}$ and as $Y^T_{t_0} \in ]Z_{t_0-}-\gep,Z_{t_0-}[$ this implies that $Y^T$ is increasing on $[t_0,h]$.
The proof of the last point in \eqref{eq:Ysaut} follows with the argument that has been used to prove \eqref{eq: corridor 1}.

One can check that conditions in \eqref{eq:Ysaut} together with the second point in \eqref{eq:after} implies that $w_s(Y^T,t_0,\delta)<2\gep$ and that concludes the proof.
It remains to prove \eqref{eq:before} and \eqref{eq:after}.

For \eqref{eq:after},
as $L_{t_0}<Z_{t_0-}-\gep$ (if $\gep$ is small enough), continuity of $L$ ensures that is still true for $t\in[t_0,t_0+\gd]$ if $\delta$ is taken small enough, 
and this yields the first point of \eqref{eq:after}.
The second one follows from uniform continuity.

We turn to \eqref{eq:before}. 
As $Z_{t_0-}\in\mathcal Z^\mathrm{n}_t$, arguing as in the proof of the second item of Lemma \ref{lem: proba zero events}, 
there exists a function $S$ defined in a neighborhood of $Z_{t_0^-}$ so that, 
in a neighborhood $B$ of $(t_0,Z_{t_0^-})$, the zeros of $u$ coincide with the graph of $S$. 
Moreover $S$ satisfies 
\beq
\label{eq: zero locally}
		 S(x) - t_0 =-\frac{1}{2}(x - Z_{t_0-})^2 + \mathcal O (|x - Z_{t_0-}|^3)
\eeq
as $x \to Z_{t_0-}$.
As we assumed that   $Z_s=L_s$ for all $t_0-\gd \leq s<t_0$ and $Z_s=R_s$ for all $t_0 \leq s \leq t_0+\gd$, it holds that, for all $(t,x)\in B$, $u(t,x)>0$ if $t >  S(x)$ and has opposite sign if $t <  S(x)$. Using \eqref{eq: zero locally}, we deduce that for $\gep>0$ small enough there exists $\gd>0$ so that for all $t_0-\gd \leq t \leq t_0$,	
	$u(t,Z_{t_0-}-\gep) > 0$ and $u(t,Z_{t_0-}) < 0$.
Moreover for $\gd>0$ small enough
	$$
		\{ (t,Z_t) : t_0 - \gd \le t < t_0 \}
		\; = \; 
		\{ (S (x),x) : Z_{t_0 - \gd} \le x < Z_{t_0 -}\},
	$$
so that from the continuity of $Z$ and \eqref{eq: zero locally} we obtain that for $\delta>0$ small enough $Z_{t_0-\gd}\in ]Z_{t_0-}-\gep,Z_{t_0-}[$.
\end{proof}


\appendix
\section{}\label{appendix}

In this appendix, we provide the needed details to understand the implications of two earlier works, \cite{huveneers} and \cite{jara_menezes}, 
{for the understanding of the process $X$ evolving in a rough potential, as described in the introduction, see \eqref{eq: edwards wilkinson} and \eqref{eq: X in rough field}.}
We can try to construct a process $X$ solving \eqref{eq: X in rough field} in three steps: 
First, we replace the velocity field $u$ by a regularized field $u^\ell$, varying smoothly in space on some length scale $\ell > 0$;
second, we define the associated process $X^\ell$;
and third, we obtain $X$ as the limit of the processes $X^\ell$ when the regularization is removed, i.e.\@ for $\ell \to 0$. 
Concretely, for $\ell > 0$, let 
\begin{equation}\label{eq: regularization appendix}
	u_\ell (t,\cdot) \; = \; P_{\ell^2} \star u (t,\cdot) ,  
\end{equation}
where the heat kernel $P$ is defined in \eqref{eq: gaussian kernel} and where $u = - \partial_x V$ with $V$ solving \eqref{eq: edwards wilkinson}. 
Let then $X^\ell$ be the solution of the Cauchy problem \eqref{eq: X in rough field} with $u^\ell$ instead of $u$, i.e.\@ $X^\ell_0 = 0$ and
\begin{equation}\label{eq: X l appendix}
	\partial_t X^\ell_t \; = \; u_\ell (t,X_t^\ell). 
\end{equation}

Let us first consider the analysis performed in \cite{huveneers}: 
{We recall the main results found there, and we explain the connection with the above problem.}
{Let $\lambda > 0$. }
In \cite{huveneers}, the process $S^\lambda$ satisfying $S^\lambda_0 = 0$ and solving 
\begin{equation}\label{eq: evolution of W appendix}
	\partial_t S_t^\lambda  \; = \; \lambda u_1(t,S_t^\lambda) , \qquad t \ge 0,
\end{equation}
is studied numerically for various values of $\lambda >0 $. 
The upshot is that, in the limit $\lambda \to 0$, and as far as numerical simulations can be reliably performed,  
\begin{equation}\label{eq: quantitative estimates appendix}
	\bE((S_t^\lambda)^2) \; \sim \; \lambda^2 t^{3/2} \quad \text{for} \quad 0 \le t \le \lambda^{-4}
	\qquad \text{and} \qquad 
	\bE((S_t^\lambda)^2) \; \sim \;  t \quad \text{for} \quad t \ge \lambda^{-4},
\end{equation}
up to possible logarithmic corrections for $t \ge \lambda^{-4}$.  

{For $\ell > 0$, we can now define a process $\tilde X^\ell$ that will have the same law as $X^\ell$ solving \eqref{eq: X l appendix}: 
For all $t\ge 0$, }
\begin{equation}\label{eq: X tilde appendix}
	\tilde X_t^\ell \; = \; \ell S^{\ell^{1/2}}_{t/\ell^2} .
\end{equation}
Indeed, since $S^\ell$ solves \eqref{eq: evolution of W appendix}, the process $\tilde X^\ell$ solves
\begin{equation}\label{eq: X tilde appendix}
	\partial_t \tilde X_t^\ell  \; = \; \ell^{-1/2} u_1 (\ell^{-2}t, \ell^{-1}\tilde X_t^\ell).
\end{equation}
With the regularization \eqref{eq: regularization appendix}, the scaling relation 
\begin{equation}\label{eq: scaling relation appendix}
	\big(u_1 (t,x)\big)_{t\ge 0, x \in \R} \; = \; \big(\ell^{1/2} u_\ell (\ell^2 t , \ell x) \big)_{t\ge 0, x \in \R}
\end{equation}
holds in law for all $\ell > 0$, as can be checked by computing the covariance of both fields. 
Therefore, we conclude from \eqref{eq: X tilde appendix} that $\tilde X^\ell = X^\ell$ in law. 
{At this point, using \eqref{eq: X tilde appendix} and the equality $\tilde X^\ell = X^\ell$ in law, 
we may reformulate \eqref{eq: quantitative estimates appendix} as: }
$$
	\bE((X_t^\ell)^2) \; \sim \; t^{3/2} \quad \text{for} \quad 0 \le t \le 1
	\qquad \text{and} \qquad 
	\bE((X_t^\ell)^2) \; \sim \;  t \quad \text{for} \quad t \ge 1. 
$$
Since these estimates do not depend on $\ell$, they make the case for the existence of a limit process $X$ solving \eqref{eq: X in rough field}.
Let us next move to the result in \cite{jara_menezes} quoted in the introduction. 
{We have already described in the main text the convergence of the processes $(W^n)_{n\ge 1}$ studied in \cite{jara_menezes}.}
Here, to make here our point, let us define a sequence of processes $U^n = (U^n_t)_{0 \le t \le T}$ 
that can reasonably be expected to behave as the processes $W^n$, 
and for which the connection with \eqref{eq: X in rough field} can be made very easily through a scaling argument. 
For $n\in \N^*$, let $U^n$ be a real valued process satisfying $U^n_0 = 0$ and solving 
\begin{equation}\label{eq: Y in rough field}
	\partial_t U_t^n  \; = \;   n u_1 (n^2t,U_t^n) \quad \text{for} \quad 0 \le t \le T.  
\end{equation}
For large values of $t$, and in the large $n$ limit, we may expect that $U^n$ and $W^n$ behave in a similar way. 
In particular, we expect the scaling {$\bE (U^n(t)^2) \sim n t^{3/2}$} to hold in this regime.

{Again, for $\ell > 0$, let us define a process $\hat X^\ell = (\hat X_t^\ell)_{0 \le t \le \ell T}$ that will turn out to have the same law as $X^\ell$
for $0 \le t \le \ell T$: 
$$
	\hat X^\ell_t \; = \; \ell U^{\ell^{-1/2}}_{t/\ell},
$$
where we have assumed that $\ell$ is such that $\ell^{-1/2}$ is an integer. 
Indeed, from \eqref{eq: Y in rough field}, we deduce that $\hat X^\ell_t$ solves 
$$
	\partial_t \hat X_t^\ell  \; = \;   \ell^{-1/2} u_{1} (\ell^{-2} t,\ell^{-1} \tilde X_t^\ell) \quad \text{for} \quad 0 \le t \le \ell T  
$$
and, by the scaling relation \eqref{eq: scaling relation appendix}, we deduce that $X^\ell = \hat X^\ell$ in law, for $t\in [0,\ell T]$. 
We observe also that $\bE ((\tilde X^\ell(t))^2) \sim t^{3/2}$ on this time interval. 
}
This brings thus some support to the validity of \eqref{eq: sub diffusive delta t}, 
{but the time interval $[0,\ell T]$ shrinks to 0 as $\ell \to \infty$,} and $\hat X^\ell$ should thus be controlled on longer time scales to reach a firm conclusion.


\end{document}